\documentclass[review,onefignum,onetabnum]{siamart250211}

\usepackage{lipsum}
\usepackage{amsfonts}
\usepackage{graphicx}
\usepackage{epstopdf}
\usepackage{algorithmic}
\usepackage{amssymb}
\usepackage[english]{babel}
\usepackage{array}
\usepackage{adjustbox}
\usepackage{relsize}
\usepackage{spverbatim}
\usepackage{listings}
\usepackage{mathrsfs}
\usepackage{derivative}
\usepackage{mathtools}
\usepackage{float}
\usepackage[section]{placeins}

\usepackage{hyperref}
\usepackage{cleveref}

\usepackage{etoolbox}

\makeatletter
\newcommand{\crefcomma}[1]{%
  \begingroup
    \def\crefcomma@sep{}%
    \forcsvlist{\crefcomma@do}{#1}%
  \endgroup
}
\newcommand{\crefcomma@do}[1]{%
  \ifx\crefcomma@sep\@empty\else,~\fi
  \cref{#1}%
  \def\crefcomma@sep{,}%
}
\newcommand{\Crefcomma}[1]{%
  \begingroup
    \def\crefcomma@sep{}%
    \forcsvlist{\Crefcomma@do}{#1}%
  \endgroup
}
\newcommand{\Crefcomma@do}[1]{%
  \ifx\crefcomma@sep\@empty\else,~\fi
  \Cref{#1}%
  \def\crefcomma@sep{,}%
}
\makeatother

\DeclareMathAlphabet{\mathpzc}{OT1}{pzc}{m}{it}

\newcolumntype{L}{>{$}l<{$}}

\ifpdf
  \DeclareGraphicsExtensions{.eps,.pdf,.png,.jpg}
\else
  \DeclareGraphicsExtensions{.eps}
\fi

\newsiamremark{hypothesis}{Hypothesis}
\crefname{hypothesis}{Hypothesis}{Hypotheses}
\newsiamthm{claim}{Claim}

\nolinenumbers

\headers{ODEs having a turning point and a double pole}{T. M. Dunster}

\title{Asymptotic expansions for solutions of differential equations having a coalescing turning point and double pole, with an application to Legendre functions}

\author{T. M. Dunster\thanks{Department of Mathematics and Statistics, San Diego State University, 5500 Campanile Drive, San Diego, CA 92182-7720, USA. 
  (\email{mdunster@sdsu.edu}, \url{https://tmdunster.sdsu.edu}).}
  }

\usepackage{amsopn}

\makeatletter
\newcommand*{\addFileDependency}[1]{
  \typeout{(#1)}
  \@addtofilelist{#1}
  \IfFileExists{#1}{}{\typeout{No file #1.}}
}
\makeatother

\nolinenumbers
\ifpdf
\hypersetup{
  pdftitle={Asymptotic expansions for solutions of differential equations having a coalescing turning point and double pole, with an application to Legendre functions},
  pdfauthor={T. M. Dunster}
}
\fi
\begin{document}
\maketitle

\begin{abstract}
The asymptotic behavior of solutions to the second-order linear differential equation $d^{2}w/dz^{2}={u^{2}f(\alpha,z)+g(z)}w$ is analyzed for a large real parameter $u$ and $\alpha\in[0,\alpha_{0}]$, where $\alpha_{0}>0$ is fixed. The independent variable $z$ ranges over a complex domain $Z$ (possibly unbounded) on which $f(\alpha,z)$ and $g(z)$ are analytic except at $z=0$, where the differential equation has a regular singular point. For $\alpha>0$, the function $f(\alpha,z)$ has a double pole at $z=0$ and a simple zero in $Z$, and as $\alpha\to 0$ the turning point coalesces with the pole. Bessel-function approximations are constructed for large $u$ involving asymptotic expansions that are uniformly valid for $z\in Z$ and $\alpha\in[0,\alpha_{0}]$. The expansion coefficients are generated by simple recursions, and explicit error bounds are obtained that simplify earlier results. As an application, uniform asymptotic expansions are derived for associated Legendre functions of large degree $\nu$, valid for complex $z$ in an unbounded domain and for order $\mu\in[0,\nu(1-\delta)]$, where $\delta>0$ is arbitrary.
\end{abstract}

\section{Introduction}
\label{sec:Introduction}

We study differential equations of the form
\begin{equation}
\label{eq01}
\frac{d^2 w}{dz^2} 
= \left\{ u^2 f(\alpha, z) + g(z) \right\} w,
\end{equation}
with a regular singularity at $z=0$ and, for $0<\alpha\leq\alpha_{0}$, a simple turning point, that is, a zero of $f(\alpha,z)$ at $z=z_{t}\in(0,\infty)$. The location of this turning point is a continuous real function of $\alpha$ and tends to $0$ as $\alpha\to 0$.

We shall suppose that equation \cref{eq01} is given in a domain $D$ of the complex $z$ plane that contains $z=0$ and $z=x_{t}$, and may be unbounded. The functions $f(\alpha,z)$ and $g(z)$ are analytic for $z\in D\setminus\{0\}$ and depend jointly continuously on $(\alpha,z)$. For each $\alpha\in[0,\alpha_{0}]$ the only singularities of $f(\alpha,z)$ and $g(z)$ occur at $z=0$. Their Laurent expansions about $z=0$ have leading terms $\tfrac14\alpha^{2}z^{-2}$ and $-\tfrac14 z^{-2}$, respectively; that is, for $0<|z|<\delta$ and $\delta>0$ sufficiently small,
\begin{equation}
\label{eq02}
f(\alpha,z)=\frac{\alpha^{2}}{4z^{2}}
+\frac{1}{z}\sum_{j=0}^{\infty}f_{j}\,z^{j},
\end{equation}
and
\begin{equation}
\label{eq03}
g(z)=-\frac{1}{4z^{2}}
+\frac{1}{z}\sum_{j=0}^{\infty}g_{j}\,z^{j}.
\end{equation}
Hence, the regular singularity at $z=0$ has exponents $\frac12 \pm\frac12\mu$, where
\begin{equation}
\label{eq04}
\mu=u\alpha.
\end{equation}

In the case $\alpha=0$ we assume that $f(0,z)$ is nonzero, has a simple pole at $z=0$, and is negative (respectively positive) for positive (respectively negative) $z\in D$. We construct asymptotic expansions with explicit error bounds in a domain $Z\subseteq D$ that contains both the pole and the turning point, with validity uniform for $0\leq\alpha\leq\alpha_{0}$. For historical background on coalescing turning points and double poles, see \cite{Boyd:1986:TPS}.

We shall apply our theory to obtain uniform asymptotic approximations for associated Legendre functions. There have been numerous numerical and asymptotic studies of associated Legendre functions for large parameters, both through diff eq and integral and difference approaches; see for example \cite{Bakaleinikov:2020:UAL,Frenzen:1990:EBU,Gil:2000:CTF,Nemes:2020:LPA,Olver:1975:LFW,Olver:1997:ASF,Shivakumar:1988:EBU,Temme:2015:AMF,Tumarkin:1959:ASO,Ursell:1984:ILP,Wang:2012:ASY}. In contrast, the present results, together with the complementary expansions in \cite{Dunster:2026:SUA}, are valid when one or both parameters are large, with coefficients that are readily computable.

Our general asymptotic theory is also applicable to Laguerre polynomials, spheroidal wave functions, Whittaker’s confluent hypergeometric functions, and Jacobi, ultraspherical, and Charlier polynomials. These special functions arise in a wide range of applications, including the scattering of particles with spin where Legendre and Jacobi (rotation) functions are employed \cite{Durand:2019:ABE}, the description of dispersive shock waves in the focusing nonlinear Schr\"{o}dinger equation via generalized Laguerre polynomials \cite{Kotlyarov:2019:DSW}, constructions of polynomial families (including Laguerre-type and higher-order Hermite polynomials) arising in Burgers-like nonlinear PDEs \cite{Dattoli:2024:HHO}, signal-processing frame expansions of band-limited signals using prolate spheroidal wave functions \cite{Hogan:2016:STS}, bound-state solution methods for the Schr\"{o}dinger equation via Laplace transforms of the confluent hypergeometric equation \cite{Nogueira:2016:LCB}, and numerical schemes for nonlinear stochastic differential equations (driven by fractional Brownian motion) based on shifted Jacobi polynomials \cite{Singh:NSB:2023}.

We begin by introducing the well-known Liouville-Green (LG) variables (see \cite[Chap. 10, Eq. (2.02)]{Olver:1997:ASF})
\begin{equation}
\label{eq05}
\xi=\int_{z_t}^{z}f^{1/2}(\alpha,v)\,dv, \quad
W(u,\alpha,\xi)=f^{1/4}(\alpha,z)\,w(u,\alpha,z),
\end{equation}
so that \cref{eq01} becomes
\begin{equation}
\label{eq06}
\frac{d^{2}W}{d \xi^{2}}
=\left\{u^{2}+\psi(\alpha,z)\right\}W,
\end{equation}
where
\begin{equation}
\label{eq07}
\psi(\alpha,z)=\frac{4f(\alpha,z)f''(\alpha,z)-5f'(\alpha,z)^2}{16f(\alpha,z)^3}
+\frac{g(z)}{f(\alpha,z)}.
\end{equation}
Classical expansions are derived from \cref{eq06} in terms of elementary functions. We use the form in which an asymptotic series appears in the exponent of an exponential \cite{Dunster:2020:LGE}, rather than the more traditional form in which an expansion multiplies an exponential function of argument $\pm u \xi$ \cite[Chap. 10]{Olver:1997:ASF}.

From \cref{eq02,eq03}, as $z\to 0$ and $\alpha>0$,
\begin{equation}
\label{eq08}
\psi(\alpha,z)=\left\{\frac{\alpha^{2}(f_{0}-g_{0})-1}
{\alpha^{4}} \right\}z+\mathcal{O}\left(z^{2}\right),
\end{equation}
while for $\alpha=0$,
\begin{equation}
\label{eq09}
\psi(0,z)=\frac{1}{8z}
+\frac{1}{16}+\mathcal{O}(z)
\quad (z \to 0).
\end{equation}
By \cref{eq07}, we see that $\psi$ is undefined at $z=z_{t}$ since $f(\alpha,z_{t})=0$. This shows that LG expansions are not valid at turning points. They can be valid at double poles, but typically not at simple poles.

To obtain expansions valid at both $z=0$ and $z=z_{t}$, following \cite[Eq. (2.1)]{Boyd:1986:TPS}, we introduce new independent and dependent variables $\zeta$ and $\tilde{W}$ via
\begin{equation}
\label{eq10}
f(\alpha, z) \left( \frac{dz}{d\zeta} \right)^2 
= \frac{\alpha^2}{4\zeta^2} - \frac{1}{4\zeta},
\end{equation}
\begin{equation}
\label{eq11}
\tilde{W}=\left(\frac{d\zeta}{dz}\right)^{1/2} w
=\left(\frac{4\zeta^{2}f(\alpha,z)}
{\zeta-\alpha^{2}}\right)^{1/4}w,
\end{equation}
and, on integrating \cref{eq10},
\begin{equation}
\label{eq12}
\int_{\alpha^2}^{\zeta} \frac{(\alpha^2-s)^{1/2}}{2s}\,ds
=\int_{z_t}^{z}f^{1/2}(\alpha,v)\,dv,
\end{equation}
with principal square roots taken. For $\alpha=0$ this reduces to
\begin{equation}
\label{eq13}
\zeta^{1/2}=\int_{0}^{z}\left\{-f(0,v)\right\}^{1/2}\,dv.
\end{equation}
Observe, from \cref{eq05,eq11},
\begin{equation}
\label{eq14}
\tilde{W}=\frac{\zeta^{1/2}}
{\left(\zeta-\alpha^{2}\right)^{1/4}}\,W.
\end{equation}

From \cref{eq02,eq10,eq13} it follows that $\zeta=\zeta(\alpha,z)$ is analytic in $z$ on $D$ for each fixed $\alpha\in[0,\alpha_{0}]$, including at $z=0$ and $z=z_{t}$. Moreover, all derivatives are continuous for $\alpha\in[0,\alpha_{0}]$, as follows from Cauchy’s integral formula for the $n$th derivatives
\begin{equation}
\label{eq15}
\zeta^{(n)}(\alpha,z)=\frac{n!}{2\pi i}
\oint_{C}\frac{\zeta(\alpha,v)}{(v-z)^{n+1}}\,dv
\quad (n=0,1,2,\dots),
\end{equation}
where $C\subset D$ is a simple positively orientated loop surrounding $z=0$ and $z=z_{t}$. The integrand is continuous for $\alpha\in[0,\alpha_{0}]$ and $v\in C$, since the contour stays away from the pole and the turning point.

Using \cref{eq07,eq11}, the differential equation \cref{eq01} becomes
\begin{equation}
\label{eq16}
\frac{d^2 \tilde{W}}{d\zeta^2} =
\left\{u^2\left( \frac{\alpha^2}{4\zeta^2} 
- \frac{1}{4\zeta} \right)
- \frac{1}{4\zeta^2}
+ \frac{\tilde{\psi}(\alpha,\zeta)}{\zeta}\right\} \tilde{W},
\end{equation}
with
\begin{multline}
\label{eq17}
\tilde{\psi}(\alpha, \zeta) 
=\frac{\zeta\left\{2\zeta'\zeta'''
-3(\zeta'')^{2}\right\}}
{4(\zeta')^{4}}
+\frac{1}{4\zeta}
+\frac{\zeta g(z)}{(\zeta')^{2}}
\\
= \frac{\zeta + 4\alpha^2}
{16\left(\alpha^2-\zeta\right)^{2}}
+ \frac{(\alpha^2-\zeta)\psi(\alpha,z)}{4\zeta}.
\end{multline}
Here, primes denote differentiation with respect to $z$. Using \crefcomma{eq03,eq12,eq15} and the first equality in \cref{eq17}, it follows, as is true for $\zeta$, that $\tilde{\psi}$ is analytic for $z\in D$ and continuous in $\alpha$ for $\alpha\in[0,\alpha_{0}]$.

In \cite{Boyd:1986:TPS} asymptotic solutions of \cref{eq16} were obtained in the form
\begin{equation}
\label{eq18}
\tilde{W}(\zeta) \sim 
\zeta^{1/2} \mathcal{C}_{\mu}(u \zeta^{1/2}) 
\sum_{s=0}^{\infty} \frac{\tilde{A}_s(\alpha, \zeta)}{u^{2s}} +
\frac{\zeta}{u} \mathcal{C}'_{\mu}(u \zeta^{1/2}) \sum_{s=0}^{\infty} \frac{\tilde{B}_s(\alpha, \zeta)}{u^{2s}},
\end{equation}
and asymptotic validity was established in certain domains containing the pole and turning point via explicit error bounds, which were obtained from the method of successive approximations. As is typical for problems with coalescing critical points, these bounds are complicated and require auxiliary weight, modulus, and phase functions that model the approximants, in this case Bessel functions. In addition, the coefficients satisfy recursion relations involving nested integrals, which are not practical for symbolic or numerical evaluation.

In this paper, we take an alternative approach: we define exact solutions in terms of two coefficient functions that replace the series in \cref{eq18}, and we analyze these functions separately. We show that each admits an asymptotic expansion whose coefficients are generated by simple recursions. The accompanying error bounds do not require auxiliary functions; they are explicit, easily computable, and involve only elementary functions. This is achieved by deriving LG expansions for the solutions of \cref{eq01} and for the Bessel approximants in \cref{eq18}, and then substituting these into the expressions for the coefficient functions.

The plan of the paper is as follows. In \cref{sec:LG} we derive LG expansions for solutions of \cref{eq16}, and then in \cref{sec:Bessel} we construct LG expansions for the Bessel approximants; these LG expansions are not valid at the turning point and are not valid at the pole when the turning point is close to it, but they are valid in a domain surrounding both points. In \cref{sec:main} we use these to construct general asymptotic expansions that are valid in an unbounded domain containing both points. In \cref{sec:Legendre} we apply the new theory to derive uniform asymptotic expansions for associated Legendre functions of large degree $\nu$, valid for complex $z$ with $|\arg z|\leq \pi/2$ and for orders $\mu$ satisfying $0\leq \mu \leq \nu(1-\delta)$, where $\delta$ is an arbitrarily small fixed positive constant.

\section{LG solutions}
\label{sec:LG}

The LG variable $\xi$ plays a central role. From \cref{eq05,eq12}, for $\alpha>0$ it is related to $\zeta$ through
\begin{equation}
\label{eq19}
\xi=\int_{\alpha^2}^{\zeta} \frac{(\alpha^2-s)^{1/2}}{2s}\,ds
=\left(\alpha^{2}-\zeta\right)^{1/2}
+\tfrac{1}{2}\alpha\ln(\zeta)
-\alpha\ln\left\{\left(\alpha^{2}
-\zeta\right)^{1/2}+\alpha\right\},
\end{equation}
with $\xi=-i\,\zeta^{1/2}$ when $\alpha=0$. The roots are chosen so that $i\xi$ is positive for $\alpha^{2} < \zeta < \infty$, and by continuity elsewhere for $\zeta$ lying in the plane cut along $(-\infty,\alpha^{2}]$. Thus, $\xi$ is negative when $\zeta$ is positive and lies above the part of the cut along $(0,\alpha^{2}]$, and it is positive when $\zeta$ is positive and lies below this part of the cut. 

\begin{figure}
 \centering
 \includegraphics[
 width=1.0\textwidth,keepaspectratio]{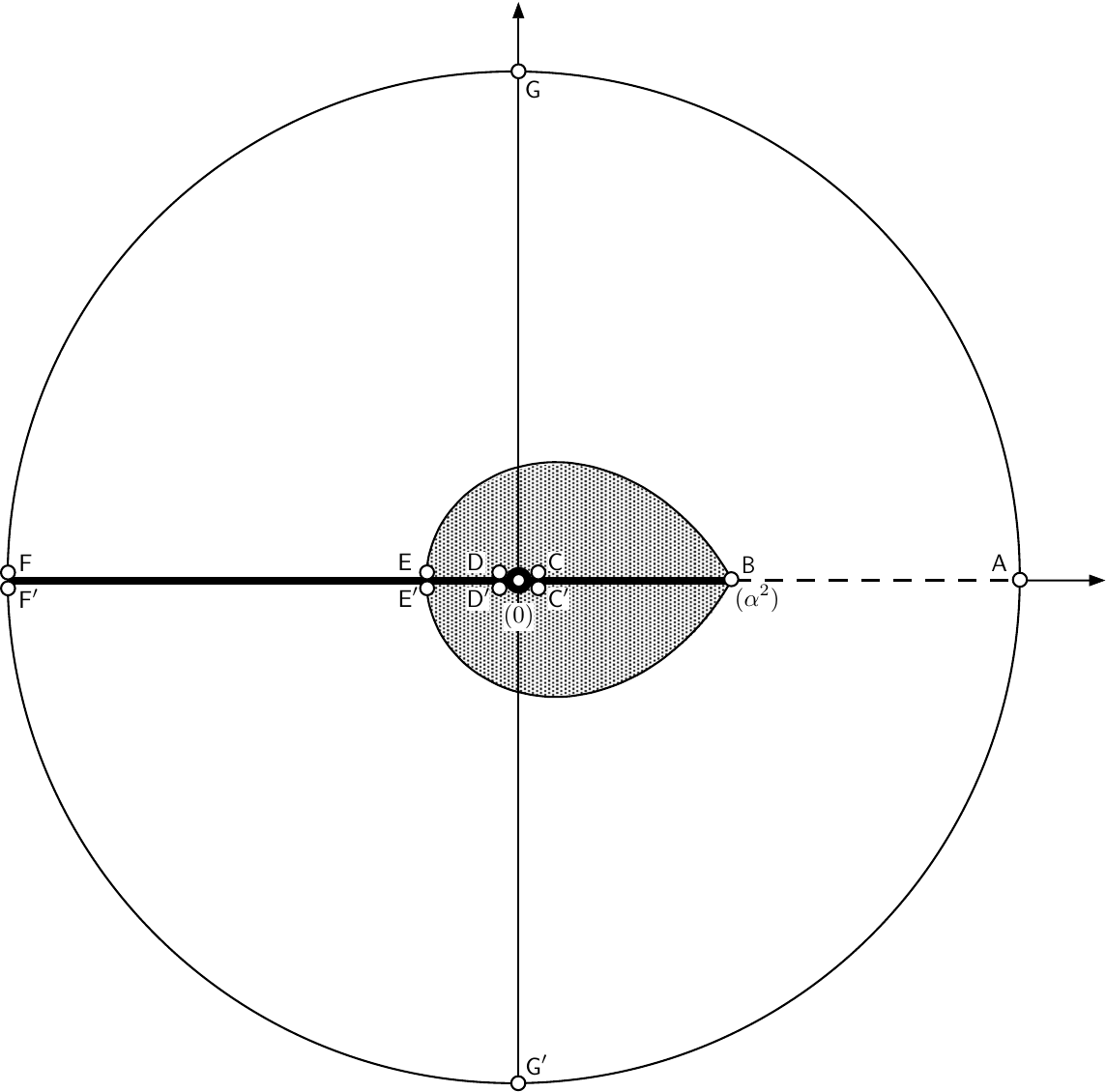}
 \caption{$\zeta$ plane, with $\Delta$ the unshaded domain.}
 \label{fig:zeta}
\end{figure}
\begin{figure}
 \centering
 \includegraphics[
 width=1.0\textwidth,keepaspectratio]{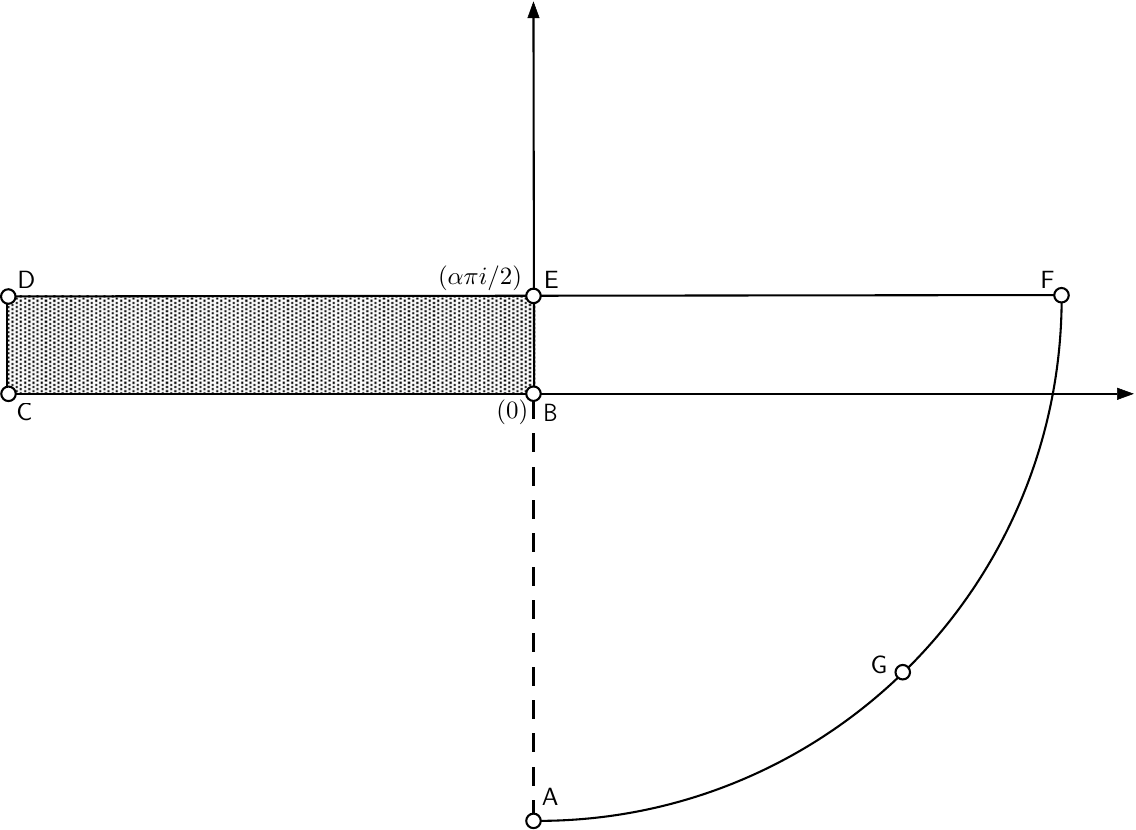}
 \caption{$\xi$ plane.}
 \label{fig:xi1}
\end{figure}
\begin{figure}
 \centering
 \includegraphics[
 width=1.0\textwidth,keepaspectratio]{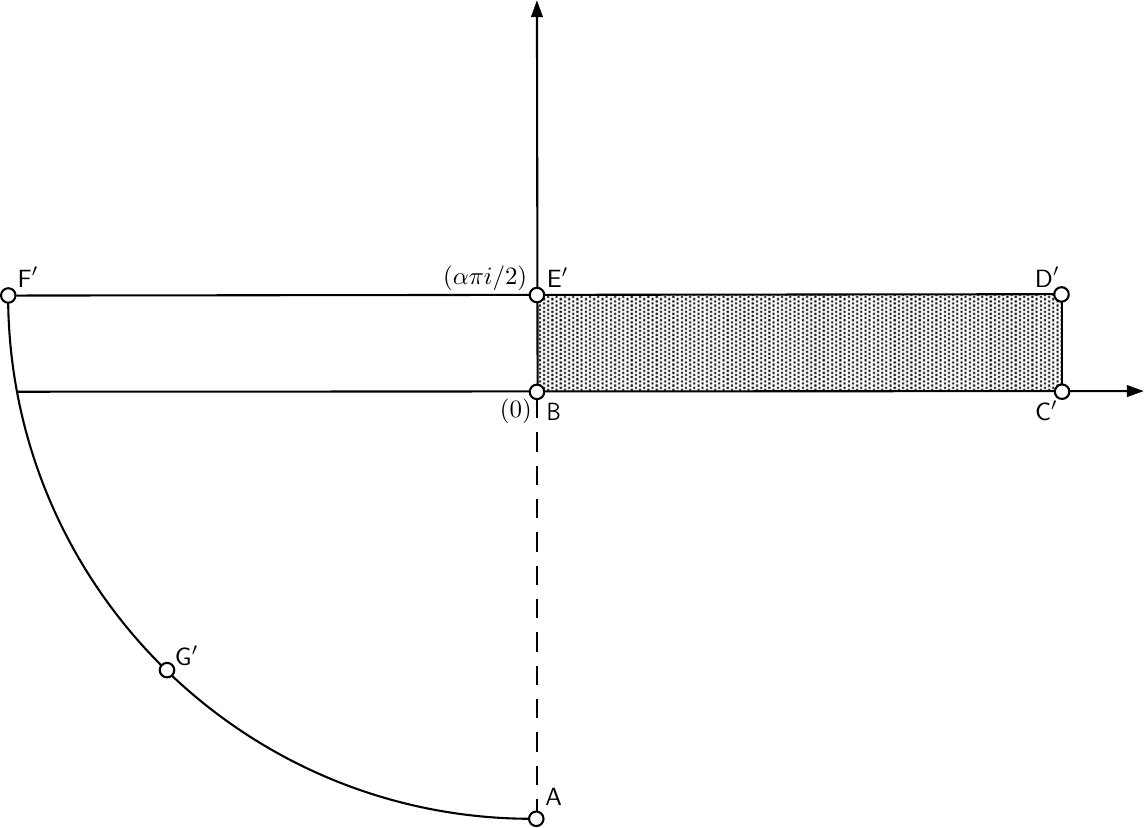}
 \caption{$\xi$ plane.}
 \label{fig:xi2}
\end{figure}

\cref{fig:zeta,fig:xi1,fig:xi2} depict the map of the $\zeta$ plane, cut along $(-\infty,\alpha^{2}]$, to the $\xi$ plane, with \cref{fig:xi1,fig:xi2} showing the images of the upper half-plane and the lower half-plane, respectively.

In establishing these, it is helpful to use the facts derived from \cref{eq19} that
\begin{equation}
\label{eq20}
\xi = -i\,\zeta^{1/2}+ \tfrac12 \alpha \pi i
+\mathcal{O}\left(\zeta^{-1/2}\right)
\quad (\zeta \to \infty),
\end{equation}
and
\begin{equation}
\label{eq21}
\xi=\pm \frac{1}{2}\alpha \ln\left(\frac{\zeta}
{4\alpha^{2}}\right) \pm \alpha+\mathcal{O}(\zeta)
\quad (\alpha>0, \, \zeta \to 0\pm i0).
\end{equation}

Denote by $\Delta$ the unshaded domain depicted in \cref{fig:zeta}. Its boundary $\mathsf{E}\mathsf{B}\mathsf{E}'$ is the level curve $\Re(\xi)=0$ and is excluded. From \cref{eq19} the points with suffixes $\mathsf{E}$ and $\mathsf{E}'$ are located at $\zeta=-\alpha^{2}\rho\pm i0$, where $\rho=0.4392288399\cdots$ is the positive root of
\begin{equation}
\label{eq22}
2(\rho + 1)^{1/2} + \ln(\rho) 
- \ln\left\{\rho+ 2(\rho + 1)^{1/2} 
+ 2 \right\}=0.
\end{equation}

Let $\zeta_{1,2}=-\infty\pm i0$, corresponding to $\xi_{1,2}=\pm\infty+\alpha\pi i/2$. Each of the corresponding points in the $z$ plane, which we denote by $z_{1,2}$, is a finite singular point or a singular point at infinity of \cref{eq01}. We construct a numerically satisfactory pair of asymptotic solutions that are recessive at these singularities and, by appropriate connection formulas, a solution that is recessive at $z=0$.
\begin{definition}
\label{defn:Xi}
Let a $\xi$ domain $\Xi$ be the set, or subset, of the union of the regions depicted in \cref{fig:xi1,fig:xi2}. It is the set of points $\xi$ satisfying
\begin{itemize}
\item[(i)] $\Im(\xi)\leq \alpha\pi/2$, except that $\xi$ does not lie on the line segment $\Re(\xi)=0$, $0\leq \Im(\xi)\leq \alpha\pi/2$;
\item[(ii)] $\tilde{\psi}(\alpha,\zeta)$ is analytic at the point $\zeta$ corresponding to $\xi$;
\item[(iii)] $\xi$ lies on a continuous directed path $\mathcal{L}$, consisting of a finite chain of $R_{2}$ arcs (as defined in \cite[Chap. 5, \S 3.3]{Olver:1997:ASF}), that runs from $\xi_{1}=\infty+\alpha\pi i/2$ to $\xi_{2}=-\infty+\alpha\pi i/2$, with points $v\in\mathcal{L}$ satisfying (i)–(ii) and $\Re(v)$ decreasing along the path.
\end{itemize}
\end{definition}
\begin{definition}
\label{defn:Z}
Let $D_{\infty}$ and $\Delta_{\infty}$ be the $z$ and $\zeta$ domains, respectively, corresponding to the $\xi$ domain $\Xi$, and let $D_{0}$ and $\Delta_{0}$ be the closed $z$ and $\zeta$ domains, respectively, corresponding to $\{\xi: 0\leq \Im(\xi)\leq \alpha\pi/2\}$ (the shaded strips in \cref{fig:xi1,fig:xi2}). Then define $Z=D_{0}\cup D_{\infty}\subseteq D$.
\end{definition}

Note that $\Delta_{0}$ is the closed bounded shaded region depicted in \cref{fig:zeta}, $\Delta$ is the complement set, and $\Delta_{\infty} \subseteq \Delta$. The LG expansions of \cref{sec:LG} will be valid for $z\in D_{\infty}$, and the LG expansions for the Bessel function approximants of \cref{sec:Bessel} will be valid for $\zeta \in \Delta$. The region $D_{\infty}$ contains the singularities $z=z_{1,2}$ but not the turning point $z_{t}$ or the pole $z=0$, and similarly $\Delta$ contains the singularities $\zeta=\zeta_{1,2}=-\infty \pm i0$ but not the turning point $\zeta=\alpha^{2}$ or the pole $\zeta=0$, both of which lie in $\Delta_{0}$.

The uniform expansions involving Bessel functions we derive will be valid for $z\in Z$, which contains the pole $z=0$, the turning point $z=z_{t}$, and the singularities $z_{1,2}$. This is achieved essentially by Cauchy’s integral formula, which allows combinations of our LG expansions to be analytically continued into the ``forbidden'' region containing the coalescing turning point and pole. 

It is important to note that allowing the pole and the turning point to be separated or to coalesce yields expansions that are valid over a wide parameter range and improves accuracy throughout the region of validity. A further advantage of our new expansions is their uniform validity at all three singularities, namely the pole $z=0$ and the finite or infinite singularities at $z=z_{1,2}$.

Instead of \cref{eq01}, we focus on the transformed equations \cref{eq06,eq16} and use them to construct LG expansions. By \cref{eq11,eq14}, these then relate to the corresponding recessive solutions of \cref{eq01}.

First, following \cite[Eqs. (1.11) and (1.15)]{Dunster:2020:LGE} we define
\begin{equation}
\label{eq23}
F_{1}(\alpha,z)=\frac12 \psi(\alpha,z), \;
F_{2}(\alpha,z)=-\frac{1}{4f^{1/2}(\alpha,z)}
\frac{d \psi(\alpha,z)}{dz},
\end{equation}
and
\begin{multline} 
\label{eq24}
F_{s+1}(\alpha,z)=-\frac{1}{2f^{1/2}(\alpha,z)}
\frac{d F_{s}(\alpha,z)}{dz}
-\frac{1}{2}\sum_{j=1}^{s-1}F_{j}(\alpha,z)
F_{s-j}(\alpha,z).
\\
\quad (s=2,3,4,\ldots).
\end{multline}
Then, from \cite[Eq. (1.14)]{Dunster:2020:LGE} we have for the odd LG coefficients
\begin{equation}
\label{eq25}
E_{2s+1}(\alpha,z)=\int_{z_{1}}^{z} 
F_{2s+1}(\alpha,v)f^{1/2}(\alpha,v)dv
+\lambda_{2s+1}(\alpha)
\quad (s=0,1,2, \ldots ),
\end{equation}
where for later ease of notation, we choose the integration constants so that
\begin{equation}
\label{eq26}
\lambda_{2s+1}(\alpha)=-\frac{1}{2}\int_{z_{1}}^{z_{2}} 
F_{2s+1}(\alpha,v)f^{1/2}(\alpha,v)dv.
\end{equation}
For the even coefficients, we set $\lambda_{2s}(\alpha)=0$.

Equation \cref{eq25} can also be used for the even LG coefficients $E_{2s}(\alpha,z)$ ($s=1,2,3,\ldots$). However, as shown in \cite{Dunster:2017:COA} using Abel's identity, these coefficients can be determined without integration by expanding the formal identity
\begin{equation}
\label{eq27}
\sum\limits_{s=1}^{\infty }\frac{E_{2s}(\alpha,z) }{u^{2s}}
= -\frac{1}{2}\ln \left\{ 1+\sum\limits_{s=1}^{\infty}
\frac{F_{2s-1}(\alpha,z) }{u^{2s}}\right\} ,
\end{equation}
in inverse powers of $u^{2}$, and then equating like powers. For example, $E_{2}(\alpha,z)=-\frac{1}{2}F_{1}(\alpha,z)$. 

From \cref{eq27,eq26,eq25} we have
\begin{equation}
\label{eq28}
E_{2s}(\alpha,z_{1})=E_{2s}(\alpha,z_{2})
=\lambda_{2s}(\alpha)=0
\quad  (s=1,2,3,\ldots),
\end{equation}
and
\begin{equation}
\label{eq29}
E_{2s+1}(\alpha,z_{1})
=-E_{2s+1}(\alpha,z_{2})
=\lambda_{2s+1}(\alpha)
\quad  (s=0,1,2,\ldots).
\end{equation}

Next, since $\tilde{\psi}(\alpha,\zeta)$ is analytic at $\zeta$, the differential equation \cref{eq16} has a regular singularity at $\zeta=0$ with exponents $\pm \tfrac12\mu+\tfrac12$, the same as \cref{eq01} at $z=0$. We regard solutions as functions of the parameter $\mu=u \alpha$, with $\alpha$ considered as a separate parameter if it appears elsewhere in the coefficients $f_{j}$ in \cref{eq02}.

From \cite[Eqs. (1.16)-(1.18)]{Dunster:2020:LGE}, denote the fundamental LG solutions of \cref{eq16}, which for convenience we consider as functions of $z$, by
\begin{multline} 
\label{eq30}
\tilde{W}_{\mu}^{(1)}(u,\alpha,z)
=\frac{\zeta^{1/2}}
{\left(\zeta-\alpha^{2}\right)^{1/4}}
\exp \left\{\sum\limits_{s=1}^{n-1}{
\frac{\lambda_{s}(\alpha)}{u^{s}}}\right\} 
\\ \times
\exp \left\{-u \xi-\frac14 \pi i
+\sum\limits_{s=1}^{n-1}{(-1)^{s}
\frac{E_{s}(\alpha,z)}{u^{s}}}
\right\} 
\left\{1+\eta_{n}^{(1)}(u,\alpha,z) \right\},
\end{multline}
and
\begin{multline} 
\label{eq31}
\tilde{W}_{\mu}^{(2)}(u,\alpha,z) 
=\frac{\zeta^{1/2}}
{\left(\zeta-\alpha^{2}\right)^{1/4}}
\exp \left\{\sum\limits_{s=1}^{n-1}{
\frac{\lambda_{s}(\alpha)}{u^{s}}}\right\} 
\\ \times
\exp \left\{u\xi+\frac14 \pi i+\sum\limits_{s=1}^{n-1}
{\frac{E_{s}(\alpha,z)}
{u^{s}}}\right\} 
\left\{1+\eta_{n}^{(2)}(u,\alpha,z) \right\}.
\end{multline}
Inclusion of the constants $\mp \pi i/4$ in the exponents is for later convenience. By \cref{eq11}, the LG solutions of \cref{eq01} are obtained from \cref{eq30,eq31} by replacing the factors $\zeta^{1/2}(\zeta-\alpha^{2})^{-1/4}$ with $f^{-1/4}(\alpha,z)$.

The relative error terms $\eta_{n}^{(1,2)}(u,\alpha,z)$ are $\mathcal{O}(u^{-n})$ as $u \to \infty$ in certain domains, and they vanish as $z \to z_{1,2}$. Before we present error bounds that reflect this, we note that, from \cref{eq04,eq20},
\begin{equation}
\label{eq32}
u\xi = \pm u\,|\zeta|^{1/2}+ \tfrac12 \mu \pi i
+\mathcal{O}\left(|\zeta|^{-1/2}\right)
\quad (\zeta \to -\infty \pm i0),
\end{equation}
and hence, from \cref{eq28,eq29,eq30,eq31},
\begin{equation}
\label{eq33}
\tilde{W}_{\mu}^{(1,2)}(u,\alpha,z)  \sim 
\zeta^{1/4} \exp \left\{-u\,|\zeta|^{1/2} 
\mp \tfrac14 (2\mu+1) \pi i \right\}
\quad (z \to z_{1,2}).
\end{equation}
This is the characteristic recessive behavior at these singularities, and also shows that both solutions are independent of $n$.

From \cref{eq33}, observe that if we replace $\mu$ by $-\mu$, the solutions $W_{-\mu}^{(1,2)}(u,\alpha,z)$ are recessive at the same singularities, and therefore must be equal to $\tilde{W}_{\mu}^{(1,2)}(u,\alpha,z)$, respectively, to within a multiplicative constant. From \cref{eq33} these constants are evident and we have
\begin{equation}
\label{eq34}
\tilde{W}_{-\mu}^{(1,2)}(u,\alpha,z)
=e^{\pm \mu \pi i}\tilde{W}_{\mu}^{(1,2)}(u,\alpha,z).
\end{equation}

Now, from \cite[Thm. 1.1]{Dunster:2020:LGE}, for $j=1,2$
\begin{equation}
\label{eq35}
\left\vert \eta_{n}^{(j)}(u,\alpha,z) \right\vert 
\leq u^{-n}
\chi_{n,j}(u,\alpha,z) \exp \left\{u^{-1}
T_{n,j}(u,\alpha,z) +u^{-n}\chi_{n,j}(u,\alpha,z) \right\},
\end{equation}
where
\begin{multline} 
\label{eq36}
\chi_{n,j}(u,\alpha,z) 
=2\int_{z_{j}}^{z}
{\left\vert F_{n}(\alpha,v) \{f(\alpha,v)\}^{1/2}dv\right\vert } 
\\ 
+\sum\limits_{s=1}^{n-1}\frac{1}{u^{s}}
 \sum\limits_{k=s}^{n-1}
\int_{z_{j}}^{z} {\left\vert 
F_{k}(\alpha,v) F_{s+n-k-1}(\alpha,v)
\{f(\alpha,v)\}^{1/2}dv \right\vert },
\end{multline}
and 
\begin{equation}
\label{eq37}
T_{n,j}(u,\alpha,z)
=4\sum\limits_{s=0}^{n-2} \frac{1}{u^{s}}
\int_{z_{j}}^{z}
\left\vert F_{s+1}(\alpha,v) 
\{f(\alpha,v)\}^{1/2}dv \right\vert.
\end{equation}
Note that we have used \cref{eq05} to change the variable in the integrals so that they are functions of $z$ rather than $\xi$.

These bounds hold where the integrals converge; thus, the integration paths must avoid the turning point $z=z_{t}$, since, by \cref{eq07,eq23,eq24}, the coefficients are not integrable at this point. At the singularities $z=z_{1,2}$ they are assumed to converge. The conditions on $f(\alpha,z)$ and $g(z)$ to ensure this are well known, for example given by \cite[Chap. 10, Thm. 4.1]{Olver:1997:ASF} when they are poles, and \cite[Chap. 10, Exercise 4.1]{Olver:1997:ASF} when they are at infinity.

At the pole $z=0$ the integrals do converge, provided that $\alpha\in(0,\alpha_{0}]$ is fixed, but we do not require LG expansions valid near this pole. Note that they are not valid at $z=0$ when $\alpha=0$, i.e. when the turning point has coalesced with the double pole to produce a simple pole (see \cref{eq09,eq23}).

Consider the Frobenius solution of \cref{eq16}, recessive at $z=\zeta=0$, given by
\begin{equation}
\label{eq38}
\tilde{W}_{\mu}^{(0)}(u,\alpha,z)
=c_{\mu}(u,\alpha)\zeta^{(\mu+1)/2}V_{\mu}(u,\alpha,z),
\end{equation}
where $V_{\mu}(u,\alpha,z)$ is analytic at $z=\zeta=0$ and normalized so that $V_{\mu}(u,\alpha,0) = 1$. We now have the linear relation
\begin{equation}
\label{eq39}
\tilde{W}_{\mu}^{(0)}(u,\alpha,z)
=\tilde{W}_{\mu}^{(1)}(u,\alpha,z)
+\gamma_{\mu}(u,\alpha)\tilde{W}_{\mu}^{(2)}(u,\alpha,z),
\end{equation}
for some constant $\gamma_{\mu}(u,\alpha)$, since all three are solutions of the same differential equation, and $\tilde{W}_{\mu}^{(1)}(u,\alpha,z)$ and $\tilde{W}_{\mu}^{(2)}(u,\alpha,z)$ are linearly independent. The constant $c_{\mu}(u,\alpha)$ is suitably chosen so that the coefficient of $\tilde{W}_{\mu}^{(1)}(u,\alpha,z)$ in \cref{eq39} is unity.

Now, from \cref{eq39},
\begin{equation}
\label{eq40}
\gamma_{\mu}(u,\alpha)
=-\frac{\mathscr{W}
\left\{\tilde{W}_{\mu}^{(1)}(u,\alpha,z),
\zeta^{(\mu+1)/2}V_{\mu}(u,\alpha,z)\right\}}
{\mathscr{W}
\left\{\tilde{W}_{\mu}^{(2)}(u,\alpha,z),
\zeta^{(\mu+1)/2}V_{\mu}(u,\alpha,z)\right\}},
\end{equation}
and the three functions in this expression are continuous in $\mu$ for fixed ordinary points $z$ and $\zeta$ of the differential equation, and for positive $u$ and $\alpha$. Hence $\gamma_{\mu}(u,\alpha)$ is likewise continuous; consequently, by \cref{eq38,eq39}, so is $c_{\mu}(u,\alpha)$.

Next, the following important connection formula will be required.
\begin{theorem}
\begin{equation}
\label{eq41}
\tilde{W}_{-\mu}(u,\alpha,z)
=e^{\mu \pi i}\tilde{W}_{\mu}^{(1)}(u,\alpha,z)
+\gamma_{\mu}(u,\alpha)e^{-\mu \pi i}\tilde{W}_{\mu}^{(2)}(u,\alpha,z),
\end{equation}
where, for arbitrary positive $n$,
\begin{equation}
\label{eq42}
\gamma_{\mu}(u,\alpha)
=1+\mathcal{O}\left(u^{-n}\right)
\quad (u \to \infty).
\end{equation}
\end{theorem}
\begin{proof}
First, $\zeta^{-(\mu+1)/2}\tilde{W}_{\mu}^{(0)}(u,\alpha,z)=c_{\mu}(u,\alpha)V_{\mu}(u,\alpha,z)$, regarded as a function of $\zeta$, is single-valued in the domain $\Delta_{\infty}$. Hence
\begin{multline} 
\label{eq43}
\lim_{\zeta \to -\infty+i0 } 
\left\{\left(\frac{|\zeta|}{\zeta}\right)^{(\mu+1)/2}
\exp\left\{-u|\zeta|^{1/2}\right\}
\tilde{W}_{\mu}^{(0)}(u,\alpha,z)\right\}
\\
=\lim_{\zeta \to -\infty-i0 } 
\left\{\left(\frac{|\zeta|}{\zeta}\right)^{(\mu+1)/2}
\exp\left\{-u|\zeta|^{1/2}\right\}
\tilde{W}_{\mu}^{(0)}(u,\alpha,z)\right\},
\end{multline}
with both limits existing as a consequence of \cref{eq30,eq31,eq32}. Thus from \cref{eq28,eq29}
\begin{equation}
\label{eq44}
\gamma_{\mu}(u,\alpha)
=\frac{1+\epsilon_{n}^{(1)}(u,\alpha)}
{1+\epsilon_{n}^{(2)}(u,\alpha)},
\end{equation}
where
\begin{equation}
\label{eq45}
\epsilon_{n}^{(1,2)}(u,\alpha)
=\lim_{\zeta \to -\infty \mp i0} 
\eta_{n}^{(1,2)}(u,\alpha,z)
=\mathcal{O}\left(u^{-n}\right)
\quad (u \to \infty).
\end{equation}
Next, from \cref{eq34,eq39}
\begin{equation}
\label{eq46}
\tilde{W}_{-\mu}(u,\alpha,z)
=e^{\mu \pi i}\tilde{W}_{\mu}^{(1)}(u,\alpha,z)
+\gamma_{-\mu}(u,\alpha)e^{-\mu \pi i}\tilde{W}_{\mu}^{(2)}(u,\alpha,z).
\end{equation}
We shall show that $\gamma_{-\mu}(u,\alpha)=\gamma_{\mu}(u,\alpha)$. Now $\zeta^{(\mu-1)/2}\tilde{W}_{-\mu}(u,\alpha,z)$, regarded as a function of $\zeta$, is single-valued in the domain $\Delta_{\infty}$, and so in a similar manner to \cref{eq43}, we have
\begin{multline} 
\label{eq47}
\lim_{\zeta \to -\infty+i0 }
\left\{\left(\frac{|\zeta|}{\zeta}\right)^{(1-\mu)/2}
\exp\left\{-u|\zeta|^{1/2}\right\}
\tilde{W}_{-\mu}(u,\alpha,z)\right\}
\\
=\lim_{\zeta \to -\infty-i0 } 
\left\{\left(\frac{|\zeta|}{\zeta}\right)^{(1-\mu)/2}
\exp\left\{-u|\zeta|^{1/2}\right\}
\tilde{W}_{-\mu}(u,\alpha,z)\right\}.
\end{multline}
From this, together with \cref{eq30,eq31,eq32,eq46}, we obtain
\begin{equation}
\label{eq48}
\gamma_{-\mu}(u,\alpha)
=\frac{1+\epsilon_{n}^{(1)}(u,\alpha)}
{1+\epsilon_{n}^{(2)}(u,\alpha)},
\end{equation}
where $\epsilon_{n}^{(1,2)}(u,\alpha)$ are the same as above, given by \cref{eq45}. Thus \cref{eq41,eq42} now follow from \cref{eq44,eq45,eq46,eq48}.
\end{proof}

\section{LG expansions for Bessel functions}
\label{sec:Bessel}
From Bessel's equation (see, for example, \cite[Eq. 10.2.1]{NIST:DLMF}), the scaled Bessel functions $\zeta^{1/2}\mathcal{C}_{\mu}(u \zeta^{1/2})$ ($\mathcal{C}=J,\,H^{(1)},\,H^{(2)}$) are readily verified to be solutions of
\begin{equation}
\label{eq49}
\frac{d^2 Y}{d\zeta^2} =
\left\{u^2\left( \frac{\alpha^2}{4\zeta^2} 
- \frac{1}{4\zeta} \right)
- \frac{1}{4\zeta^2}\right\}Y,
\end{equation}
recalling that $\mu = u \alpha$. This is the comparison equation for \cref{eq16}. Our aim is to obtain LG expansions for certain fundamental solutions of \cref{eq49} and their derivatives.

In the next section we regard $\xi$ and $\zeta$ as functions of $z$ through \cref{eq05,eq11,eq12}; here we consider only $\xi$ as a function of $\zeta$ via \cref{eq19}.

We express \cref{eq49} in the form $Y''=\{u^{2}\hat{f}+\hat{g}\}Y$, where
\begin{equation}
\label{eq50}
\hat{f}(\alpha,\zeta)
=\frac{\alpha^2}{4\zeta^2} - \frac{1}{4\zeta},
\quad
\hat{g}(\zeta)=- \frac{1}{4\zeta^2},
\end{equation}
and again apply the Liouville transformation to \cref{eq49}, which yields
\begin{multline} 
\label{eq51}
\xi=-i\int_{\alpha^{2}}^{\zeta}\{-\hat{f}(\alpha,s)\}^{1/2}ds
=-i\int_{\alpha^{2}}^{\zeta}
\frac{\left(s-\alpha^{2}\right)^{1/2}}{2s}ds
\\
=-i\left(\zeta-\alpha^{2}\right)^{1/2}
+i \alpha \arccos\left(\alpha \zeta^{-1/2}\right).
\end{multline}
Note that we choose the lower integration limits in \cref{eq51} so that $\xi=0$ at the turning point $\zeta=\alpha^{2}$. Then, in the usual manner, let
\begin{equation}
\label{eq52}
Y(u,\alpha,\zeta)
=\hat{f}^{-1/4}(\alpha,\zeta)\hat{W}(u,\alpha,\xi),
\end{equation}
and \cref{eq49} transforms to
\begin{equation}
\label{eq53}
\frac{d^{2}\hat{W}}{d \xi^{2}}
=\left\{u^{2}+\hat{\psi}(\alpha,\zeta)\right\}\hat{W},
\end{equation}
where
\begin{equation}
\label{eq54}
\hat{\psi}(\alpha,\zeta)=\frac{\zeta 
\left( \zeta+4 \alpha^{2}\right)}
{4 \left(\zeta-\alpha^{2}\right)^{3}}.
\end{equation}

The LG coefficients are derived as in \cref{sec:LG} (see \cite{Dunster:2020:LGE}), with $z$ replaced by $\zeta$ and $f,\,g,\,\psi$ replaced by $\hat{f},\,\hat{g},\,\hat{\psi}$. To facilitate the integrations, define
\begin{equation}
\label{eq55}
\hat{\beta}=\left(\alpha^{2}-\zeta\right)^{1/2},
\end{equation}
where the branch of the square root is such that $i\hat{\beta}$ is positive for $\alpha^{2} < \zeta < \infty$, and continuous in the $\zeta$ plane cut along $(-\infty,\alpha^{2}]$. Thus $\hat{\beta}$ is positive when real above the cut and negative when real below it. Then, from \cref{eq50,eq51,eq55}, we have
\begin{equation}
\label{eq56}
\frac{d\hat{\beta}}{d\xi}
=\frac{\hat{\beta}^{2}-\alpha^{2}}
{\hat{\beta}^{2}}.
\end{equation}
Hence, from \cref{eq54,eq55,eq56} and \cite[Eqs. (1.10)–(1.12)]{Dunster:2020:LGE}, we obtain the LG coefficients, denoted $e_{s}(\alpha,\hat{\beta})$ ($s=1,2,3,\dots$)
\begin{equation}
\label{eq57}
e_{1}(\alpha,\hat{\beta})
=\frac{3\hat{\beta}^{2}-5\alpha^{2}}
{24\hat{\beta}^{3}},
\end{equation}
\begin{equation}
\label{eq58}
e_{2}(\alpha,\hat{\beta})
=\frac{\left(\hat{\beta}^{2}-\alpha^{2}\right)
\left(\hat{\beta}^{2}-5\alpha^{2}\right)}{16\hat{\beta}^{6}},
\end{equation}
with subsequent ones also Laurent polynomials in $\hat{\beta}$, satisfying the simple recursion
\begin{multline}
\label{eq59}
e_{s+1}(\alpha,\hat{\beta})=
\frac{\alpha^{2}-\hat{\beta}^{2}}
{2\hat{\beta}^{2}}
\pdv{e_{s}(\alpha,\hat{\beta})}{\hat{\beta}}
\\
+\frac{1}{2}\int_{\infty}^{\hat{\beta}}
\frac{\alpha^{2}-p^{2}}{p^{2}}
\sum\limits_{j=1}^{s-1}
\pdv{e_{j}(\alpha,p)}{p}
\pdv{e_{s-j}(\alpha,p)}{p} dp
\quad (s=2,3,4,\dots).
\end{multline}
We have chosen the integration constants so that the coefficients vanish at $\hat{\beta}=\infty$ ($\zeta=\infty$).

Next, from the well-known behavior of the Hankel functions at infinity \cite[Eqs. 10.17.5 and 10.17.6]{NIST:DLMF}, we have
\begin{equation}
\label{eq60}
H_{\mu}^{(1,2)}\left(u\zeta^{1/2}\right)
\sim \sqrt{\frac{2}{\pi u}}\,
\frac{1}{\zeta^{1/4}}
\exp\left\{\pm i\left(u\zeta^{1/2}-\frac12 \mu \pi
-\frac14 \pi\right)\right\},
\end{equation}
as $\zeta \to \infty$ with $|\arg(\zeta)| \leq \pi$. This shows that they are recessive at $\zeta=\zeta_{1,2}=-\infty \pm i0$. Thus, letting $\zeta \to \zeta_{1,2}$ and using \cref{eq20,eq60}, we obtain the desired LG expansions
\begin{multline} 
\label{eq61}
H_{\mu}^{(1)}\left(u\zeta^{1/2}\right)
=\sqrt{\frac{2}{\pi u}}\,
\frac{1}{\left(\zeta-\alpha^{2}\right)^{1/4}}
\exp \left\{-u\xi-\frac14 \pi i+\sum\limits_{s=1}^{n-1}
{(-1)^{s}\frac{e_{s}(\alpha,\hat{\beta})}{u^{s}}}\right\} 
\\ \times
\left\{1+\hat{\eta}_{n}^{(1)}(u,\alpha,\zeta) \right\},
\end{multline}
and
\begin{multline} 
\label{eq62}
H_{\mu}^{(2)}\left(u\zeta^{1/2}\right)
=\sqrt{\frac{2}{\pi u}}\,
\frac{1}{\left(\zeta-\alpha^{2}\right)^{1/4}}
\exp \left\{u\xi+\frac14 \pi i+\sum\limits_{s=1}^{n-1}
{\frac{e_{s}(\alpha,\hat{\beta})}{u^{s}}}\right\} 
\\ \times
\left\{1+\hat{\eta}_{n}^{(2)}(u,\alpha,\zeta) \right\}.
\end{multline}
Error bounds are obtained similarly to \cref{eq35,eq36,eq37}. Using \cref{eq53} together with \cite[Thm.~1.1]{Dunster:2020:LGE}, we have
\begin{multline} 
\label{eq63}
\left\vert \hat{\eta}_{n}^{(1,2)}(u,\alpha,\zeta) \right\vert 
\leq |u|^{-n}\omega_{n}^{(1,2)}(u,\alpha,\hat{\beta}) 
\\ \times
\exp \left\{|u|^{-1}\varpi_{n}^{(1,2)}(u,\alpha,\hat{\beta}) 
+|u|^{-n}\omega_{n}^{(1,2)}(u,\alpha,\hat{\beta}) \right\},
\end{multline}
where
\begin{multline} 
\label{eq64}
\omega_{n}^{(1,2)}(u,\alpha,\hat{\beta}) 
=2\int_{\pm \infty}^{\hat{\beta}}
{\left\vert
\frac{\partial e_{n}(\alpha,p)}{\partial p} d p
\right\vert } 
\\ 
+\sum\limits_{s=1}^{n-1}\frac{1}{|u|^{s}}
 \sum\limits_{k=s}^{n-1}
\int_{\pm \infty}^{\hat{\beta}}
{\left\vert \frac{\alpha^{2}p-1}{p^{2}}
\frac{\partial e_{k}(\alpha,p)}{\partial p}
\frac{\partial e_{s+n-k-1}(\alpha,p)}{\partial p} dp
\right\vert},
\end{multline}
and
\begin{equation} 
\label{eq65}
\varpi_{n}^{(1,2)}(u,\alpha,\hat{\beta}) =4\sum\limits_{s=0}^{n-2}
\frac{1}{|u|^{s}}
\int_{\pm \infty}^{\hat{\beta}}
\left\vert
\frac{\partial e_{s+1}(\alpha,p)}{\partial p} dp 
\right\vert .
\end{equation}
The lower integration limits $\hat{\beta} =\pm \infty$ are chosen to correspond to $\zeta = -\infty \pm i0$ ($\xi = \pm \infty + \alpha \pi i/2$). Integration is taken along paths on which $\Re(\xi)$ is monotone and bounded away from $\xi =0$ (where $\zeta=\alpha^2$); by \cref{eq55,eq57,eq58,eq59} the integrals diverge there.

The regions of validity both include $\Delta$ (\cref{fig:zeta}). Certain parts of the shaded region may also be included, but details are not required.

We next require similar expansions for the derivatives of the Hankel-function approximants. We obtain them by noting, from \cref{eq49}, that
\begin{equation}
\label{eq66}
\tilde{Y}(u,\alpha,\zeta)
=\frac{\zeta^{3/2}}{\left(\zeta-\alpha^{2}\right)^{1/2}}
\frac{\partial \mathcal{C}_{\mu}(u \zeta^{1/2})}
{\partial \zeta},
\quad (\mathcal{C}=J,\,H^{(1)},\,H^{(2)}),
\end{equation}
are solutions of
\begin{equation}
\label{eq67}
\frac{d^2 \tilde{Y}}{d\zeta^2} =
\left\{u^2\left( \frac{\alpha^2}{4\zeta^2} 
- \frac{1}{4\zeta} \right)
+\frac{\alpha^{2}\left(4\zeta-\alpha^{2}\right)}{4\left(\zeta-\alpha^{2}\right)^{2}\zeta^2}\right\}\tilde{Y}.
\end{equation}
With $\xi$ as in \cref{eq51} and 
\begin{equation}
\label{eq68}
\tilde{Y}(u,\alpha,\zeta)
=\hat{f}^{-1/4}(\alpha,\zeta)\tilde{W}(u,\alpha,\xi),
\end{equation}
we obtain from \cref{eq67}
\begin{equation}
\label{eq69}
\frac{d^{2}\tilde{W}}{d \xi^{2}}
=\left\{u^{2}+\tilde{\psi}(\alpha,\zeta)\right\}\tilde{W},
\end{equation}
where in place of \cref{eq54} we have
\begin{equation}
\label{eq70}
\tilde{\psi}(\alpha,\zeta)=-\frac{\zeta 
\left(3\zeta+4 \alpha^{2}\right)}
{4 \left(\zeta-\alpha^{2}\right)^{3}}.
\end{equation}

We proceed analogously to \cref{eq61,eq62}, using \cite[Eqs. 10.17.11 and 10.17.12]{NIST:DLMF}
\begin{equation}
\label{eq71}
\frac{\partial}{\partial \zeta}H_{\mu}^{(1,2)}\left(u\zeta^{1/2}\right)
\sim \sqrt{\frac{u}{2\pi}}\,
\frac{1}{\zeta^{3/4}}
\exp\left\{\pm i\left(u\zeta^{1/2}
-\frac12 \mu \pi+\frac14 \pi\right)\right\},
\end{equation}
as $\zeta \to \infty$ with $|\arg(\zeta)| \leq \pi$. As a result, we obtain
\begin{multline} 
\label{eq72}
\frac{\partial}{\partial \zeta}H_{\mu}^{(1)}
\left(u\zeta^{1/2}\right)
=\sqrt{\frac{u}{2\pi}}\,
\frac{\left(\zeta-\alpha^{2}\right)^{1/4}}{\zeta}
\\ \times
\exp \left\{-u\xi+\frac14 \pi i+\sum\limits_{s=1}^{n-1}
{(-1)^{s}\frac{\tilde{e}_{s}(\alpha,\hat{\beta})}{u^{s}}}\right\} 
\left\{1+\tilde{\eta}_{n}^{(1)}(u,\alpha,\zeta) \right\},
\end{multline}
and
\begin{multline} 
\label{eq73}
\frac{\partial}{\partial \zeta}H_{\mu}^{(2)}\left(u\zeta^{1/2}\right)
=\sqrt{\frac{u}{2\pi}}\,
\frac{\left(\zeta-\alpha^{2}\right)^{1/4}}{\zeta}
\\ \times
\exp \left\{u\xi-\frac14 \pi i+\sum\limits_{s=1}^{n-1}
{\frac{\tilde{e}_{s}(\alpha,\hat{\beta})}{u^{s}}}\right\} 
\left\{1+\tilde{\eta}_{n}^{(2)}(u,\alpha,\zeta) \right\},
\end{multline}
where
\begin{equation}
\label{eq74}
\tilde{e}_{1}(\alpha,\hat{\beta})
=-\frac{9\hat{\beta}^{2}-7\alpha^{2}}{24\hat{\beta}^{3}},
\end{equation}
\begin{equation}
\label{eq75}
\tilde{e}_{2}(\alpha,\hat{\beta})
=-\frac{\left(\hat{\beta}^{2}-\alpha^{2}\right)
\left(3\hat{\beta}^{2}-7\alpha^{2}\right)}
{16\hat{\beta}^{6}},
\end{equation}
with the remaining coefficients satisfying the same recursion \cref{eq59} as $e_{s}(\alpha,\hat{\beta})$.

The regions of validity are the same as for the corresponding non-derivative functions; in particular, both contain $\Delta$. The error bounds \cref{eq63} again are applicable, with $e$ replaced by $\tilde{e}$ in \cref{eq64,eq65}.

\section{Main expansions}
\label{sec:main}
Based on \cref{eq18}, we define functions $A_{\mu}(u,\alpha,z)$ and $B_{\mu}(u,\alpha,z)$, which will be slowly varying in $u$, by expressing the LG solutions \cref{eq30,eq31} of \cref{eq16} in the form
\begin{equation}
\label{eq76}
\tilde{W}_{\mu}^{(1)}(u,\alpha,z)
=\zeta^{1/2}H_{\mu}^{(1)}\left(u\zeta^{1/2}\right)A_{\mu}(u,\alpha,z)
+\frac{\zeta^{3/2}}{u}\frac{\partial
H_{\mu}^{(1)}\left(u\zeta^{1/2}\right)}{\partial \zeta}
B_{\mu}(u,\alpha,z),
\end{equation}
and
\begin{multline}
\label{eq77}
\gamma_{\mu}(u,\alpha)\tilde{W}_{\mu}^{(2)}(u,\alpha,z)
=\zeta^{1/2}H_{\mu}^{(2)}\left(u\zeta^{1/2}\right)A_{\mu}(u,\alpha,z)
\\
+\frac{\zeta^{3/2}}{u}\frac{\partial 
H_{\mu}^{(2)}\left(u\zeta^{1/2}\right)}
{\partial \zeta}
B_{\mu}(u,\alpha,z).
\end{multline}
The constant $\gamma_{\mu}(u,\alpha)$ is the one that appears in the connection equation \cref{eq39}. As we shall show, its inclusion in \cref{eq77} ensures that both $A_{\mu}(u,\alpha,z)$ and $B_{\mu}(u,\alpha,z)$ are analytic at $z=0$. Note that the LG solutions on the LHS of these identities are recessive at the same singularities as the corresponding Hankel functions on the RHS, namely $\xi=\xi_{1,2}$ (correspondingly $\zeta = \zeta_{1,2}$).

Next, we solve the pair of equations for \cref{eq76,eq77} for $A_{\mu}(u,\alpha,z)$ and $B_{\mu}(u,\alpha,z)$, with the aid of the Hankel function Wronskian \cite[Eq. 10.5.5]{NIST:DLMF}. This yields the explicit representations
\begin{multline}
\label{eq78}
A_{\mu}(u,\alpha,z)=\frac{i\pi \zeta^{1/2}}{2}
\left\{\tilde{W}_{\mu}^{(1)}(u,\alpha,z)\frac{\partial 
H_{\mu}^{(2)}\left(u\zeta^{1/2}\right) }
{\partial \zeta}
\right.
\\
\left.
-\gamma_{\mu}(u,\alpha)
\tilde{W}_{\mu}^{(2)}(u,\alpha,z)\frac{\partial 
H_{\mu}^{(1)}\left(u\zeta^{1/2}\right)}
{\partial \zeta}
\right\},
\end{multline}
and
\begin{multline}
\label{eq79}
B_{\mu}(u,\alpha,z)=-\frac{i\pi u }{2 \zeta^{1/2}}
\left\{\tilde{W}_{\mu}^{(1)}(u,\alpha,z)
H_{\mu}^{(2)}\left(u\zeta^{1/2}\right) 
\right.
\\
\left.
-\gamma_{\mu}(u,\alpha)
\tilde{W}_{\mu}^{(2)}(u,\alpha,z)
H_{\mu}^{(1)}\left(u\zeta^{1/2}\right)
\right\}.
\end{multline}

Now, from \cref{eq30,eq31,eq61,eq62,eq72,eq73,eq78,eq79}, we obtain the expansions
\begin{multline} 
\label{eq80}
A_{\mu}(u,\alpha,z) \sim 
\sqrt{\frac{u \pi}{2}}\exp \left\{\sum\limits_{s=0}^{\infty}
\frac{\lambda_{2s+1}(\alpha)}{u^{2s+1}}\right\}
\\ \times
\exp \left\{ \sum\limits_{s=1}^{\infty}
\frac{\tilde{\mathcal{E}}_{2s}(a,z) }{u^{2s}}\right\} 
\cosh \left\{ \sum\limits_{s=0}^{\infty}
\frac{\tilde{\mathcal{E}}_{2s+1}(a,z)}
{u^{2s+1}}\right\},
\end{multline}
and
\begin{multline} 
\label{eq81}
B_{\mu}(u,\alpha,z) 
\sim \frac{\sqrt{2\pi u}}{\hat{\beta}}
\exp \left\{\sum\limits_{s=0}^{\infty}
\frac{\lambda_{2s+1}(\alpha)}{u^{2s+1}}\right\}
\\ \times
\exp \left\{ \sum\limits_{s=1}^{\infty}
\frac{\mathcal{E}_{2s}(a,z) }{u^{2s}}\right\} 
\sinh \left\{ \sum\limits_{s=0}^{\infty}
\frac{\mathcal{E}_{2s+1}(a,z)}
{u^{2s+1}}\right\},
\end{multline}
as $u \to \infty$ which are uniformly valid for $z \in D_{\infty}$, where $\hat{\beta}$ is given by \cref{eq55},
\begin{equation} 
\label{eq82}
\tilde{\mathcal{E}}_{s}(a,z)
=E_{s}(\alpha,z)+(-1)^{s}\tilde{e}_{s}(\alpha,\hat{\beta}),
\end{equation}
and
\begin{equation} 
\label{eq83}
\mathcal{E}_{s}(a,z)
=E_{s}(\alpha,z)+(-1)^{s}e_{s}(\alpha,\hat{\beta}).
\end{equation}

We then expand the exponentials and hyperbolic functions in \cref{eq80,eq81} whose series have coefficients dependent on $z$ as regular expansions in inverse powers of $u$. As a result, we obtain the asymptotic expansions as $u \to \infty$
\begin{equation} 
\label{eq84}
A_{\mu}(u,\alpha,z) \sim 
\sqrt{\frac{u \pi}{2}}\exp \left\{\sum\limits_{s=0}^{\infty}
\frac{\lambda_{2s+1}(\alpha)}{u^{2s+1}}\right\}
\left\{1 + \sum_{s=1}^{\infty}
\frac{\mathrm{A}_{s}(\alpha,z)}{u^{2s}}\right\},
\end{equation}
and
\begin{equation} 
\label{eq85}
B_{\mu}(u,\alpha,z) \sim 
\sqrt{\frac{2\pi}{u}} \exp \left\{\sum\limits_{s=0}^{\infty}
\frac{\lambda_{2s+1}(\alpha)}{u^{2s+1}}\right\}
\sum_{s=0}^{\infty}
\frac{\mathrm{B}_{s}(\alpha,z)}{u^{2s}},
\end{equation}
which are, at this stage, known to be valid for $z \in Z_{\infty}$. 

As noted in \cite[Remark 1]{Dunster:2025:SAR} for similar expansions, $\mathrm{A}_{s}(\alpha,z)=\tilde{q}_{2s}(\alpha,z)$ ($s=1,2,3,\ldots$) and $\mathrm{B}_{s}(\alpha,z)=\hat{\beta}^{-1} q_{2s+1}(\alpha,z)$ ($s=0,1,2,\ldots$), where for $s=1,2,3,\ldots$
\begin{equation}
\label{eq86}
\tilde{q}_{s}(\alpha,z)=\tilde{\mathcal{E}}_{s}(\alpha,z)
+\frac{1}{s}\sum_{j=1}^{s-1}j
\tilde{\mathcal{E}}_{j}(\alpha,z)\tilde{q}_{s-j}(\alpha,z),
\end{equation}
and
\begin{equation}
\label{eq87}
q_{s}(\alpha,z)=\mathcal{E}_{s}(\alpha,z)
+\frac{1}{s}\sum_{j=1}^{s-1}j
\mathcal{E}_{j}(\alpha,z)q_{s-j}(\alpha,z),
\end{equation}
with both sums being empty when $s=1$. The expansions \cref{eq84,eq85} are certainly valid for $z \in Z_{\infty}$, the LG region of validity. We now show that they also hold in $Z_{0}$, and hence in $Z=Z_{0} \cup Z_{\infty}$, in particular at the turning point and the pole.

To do so, we first prove that $A_{\mu}(u,\alpha,z)$ and $B_{\mu}(u,\alpha,z)$ are analytic in $Z$. To this end, using \cref{eq39,eq41,eq76,eq77} and \cite[Eqs.~10.4.4 and 10.4.6]{NIST:DLMF}, we express the Frobenius solutions $\tilde{W}_{\pm\mu}(u,\alpha,z)$ similarly as
\begin{equation}
\label{eq88}
\tilde{W}_{\pm\mu}(u,\alpha,z)
=2\zeta^{1/2}J_{\pm\mu}\left(u\zeta^{1/2}\right)A_{\mu}(u,\alpha,z)
+\frac{2\zeta^{3/2}}{u}\frac{\partial J_{\pm\mu}\left(u\zeta^{1/2}\right)}{\partial \zeta}
B_{\mu}(u,\alpha,z).
\end{equation}
Then, on solving this pair, we obtain the alternative representations
\begin{multline}
\label{eq89}
A_{\mu}(u,\alpha,z)=-\frac{\pi \zeta^{1/2}}{4\sin(\mu \pi)}
\left\{\tilde{W}_{\mu}^{(0)}(u,\alpha,z)\frac{\partial 
J_{-\mu}\left(u\zeta^{1/2}\right) }
{\partial \zeta}
\right. \\ \left.
-\tilde{W}_{-\mu}(u,\alpha,z)\frac{\partial 
J_{\mu}\left(u\zeta^{1/2}\right)}
{\partial \zeta}
\right\},
\end{multline}
and
\begin{multline}
\label{eq90}
B_{\mu}(u,\alpha,z)=\frac{\pi u }{4 \zeta^{1/2} \sin(\mu \pi)}
\left\{\tilde{W}_{\mu}^{(0)}(u,\alpha,z)
J_{-\mu}\left(u\zeta^{1/2}\right)
\right. \\ \left.
-\tilde{W}_{-\mu}(u,\alpha,z)
J_{\mu}\left(u\zeta^{1/2}\right)
\right\}.
\end{multline}
Expanding the Frobenius series for all four functions on the RHS of \cref{eq89,eq90} about $z=\zeta=0$ (see \cref{eq38} and \cite[Eq.~10.2.]{NIST:DLMF}) shows that $A_{\mu}(u,\alpha,z)$ and $B_{\mu}(u,\alpha,z)$ are analytic at this point for noninteger $\mu$; for integer $\mu$, this also holds by taking limits of the resulting Maclaurin series. Analyticity elsewhere for $z \in Z$ follows from analytic continuation.

The coefficients $\mathrm{A}_{s}(\alpha,z)$ and $\mathrm{B}_{s}(\alpha,z)$ are analytic in $Z$, except possibly for isolated singularities at $z=0$ and $z=z_{t}$. They do not have branch points at these singularities, because the LHS of \cref{eq84,eq85} are analytic, and hence single-valued, in $Z$.

We have established that the coefficient functions possess the asymptotic expansions \cref{eq84,eq85} with analytic coefficients in $Z_{\infty}$ (which surrounds but does not include the isolated singularities $z=0$ and $z=z_{t}$), and that both functions are analytic in $Z$. 

We now use the following theorem. Its proof is a routine modification of that of \cite[Thm. 3.1]{Dunster:2021:NKF} for a single isolated singularity; see also \cite[Thm. 4.2]{Dunster:2025:LCT}.

\begin{theorem} 
\label{thm:nopoles}
Let $0<\rho_{0}<\rho_{1}$, $u>0$, $z_{0}\in \mathbb{C}$, $H(u,z)$ be an analytic function of $z$ in the open disk $D_{1}=\{z:\, |z-z_{0}|<\rho_{1}\}$, and $h_{s}(z)$ ($s=0,1,2,\ldots$) be a sequence of functions that are analytic in the annulus $R=\{z:\, \rho_{0}<|z-z_{0}|<\rho_{1}\}$. If $H(u,z)$ is known to possess the asymptotic expansion
\begin{equation} 
\label{eq91}
H(u,z) \sim \sum\limits_{s=0}^{\infty}\frac{h_{s}(z)}{u^{s}}
\quad (u \rightarrow \infty),
\end{equation}
in $R$, then each $h_{s}(z)$ is analytic in $D_{1}$ (except possibly having removable singularities there), and the expansion \cref{eq91} holds for all $z \in D_{1}$.
\end{theorem}

We apply this by choosing $z_{0}=z_{t}/2$, $\rho_{0}=z_{0}+\delta$, $\rho_{1}=z_{0}+2\delta$, with $\delta>0$ arbitrarily small, and, if necessary, taking $\alpha$ sufficiently small so that the annulus $R$ lies in $Z_{\infty}$. This extends the expansions \cref{eq84,eq85} to $Z_{0} \subset D_{1}$. Hence, the expansions hold as $u \to \infty$, uniformly for $z \in Z = Z_{\infty} \cup Z_{0}$ and $0 \leq \alpha \leq \alpha_{0}$. Moreover, by the theorem the coefficients $\mathrm{A}_{s}(\alpha,z)$ and $\mathrm{B}_{s}(\alpha,z)$ are analytic for $z \in Z$ (treating removable singularities at the turning point as points of analyticity by continuity).

Let us summarize our results for the solutions of the differential equation \cref{eq01}, assuming conditions on $f(\alpha,z)$ and $g(z)$ as given in \cref{sec:Introduction}, $z \in Z$ as defined by \cref{defn:Z}, and $0 \leq \alpha \leq \alpha_{0}$. The parameter $\mu$ and the LG variables $\xi$ and $\zeta$ are defined by \cref{eq04,eq05,eq19}. It is assumed that \cref{eq01} has singularities at $z=0$ and $z=z_{1,2}$ corresponding to $\xi=0, \pm\infty+\alpha\pi i/2$ respectively ($\zeta=0, -\infty \pm i0$). In addition, for $0 < \alpha \leq \alpha_{0}$ the differential equation has a turning point at $z_{t} \in Z \cap [0,\infty)$, with $z_{t} \to 0$ as $\alpha \to 0$.

In the following $E_{s}(\alpha,z)$ and $\lambda_{s}(\alpha)$ are given by \cref{eq07,eq23,eq24,eq25,eq26,eq27}, $e_{s}(\alpha,\hat{\beta})$ and $\tilde{e}_{s}(\alpha,\hat{\beta})$ are given by \cref{eq55,eq57,eq58,eq59,eq74,eq75}, and $\mathcal{E}_{s}(a,z)$ and $\tilde{\mathcal{E}}_{s}(a,z)$ are given by \cref{eq82,eq83}. Our main result then reads as follows.
\begin{theorem}
\label{thm:main}
For $u>0$ and $\mu \geq 0$, the differential equation \cref{eq01} has solutions $w_{\mu}^{(j)}(u,\alpha,z)$ ($j=0,1,2$) that are recessive at $z=z_{j}$ (where $z_{0}=0$), with the properties
\begin{equation}
\label{eq92}
w_{\mu}^{(0)}(u,\alpha,z)=\mathcal{O}\left(z^{(\mu+1)/2}\right)
\quad (z \to 0),
\end{equation}
and
\begin{equation}
\label{eq93}
w_{\mu}^{(1,2)}(u,\alpha,z) \sim f^{-1/4}(\alpha,z)
e^{\mp u \xi \mp \pi i/4}
\quad (z \to z_{1,2}).
\end{equation}
These satisfy the connection relation
\begin{equation}
\label{eq94}
w_{\mu}^{(0)}(u,\alpha,z)=w_{\mu}^{(1)}(u,\alpha,z)
+\gamma_{\mu}(u,\alpha)w_{\mu}^{(2)}(u,\alpha,z),
\end{equation}
where, for all $n>0$
\begin{equation}
\label{eq95}
\gamma_{\mu}(u,\alpha)=\gamma_{-\mu}(u,\alpha)
=1+\mathcal{O}\left(u^{-n}\right)
\quad (u \to \infty).
\end{equation}
Moreover, they can be expressed in the form
\begin{multline}
\label{eq96}
\gamma_{\mu}^{(j)}(u,\alpha)w_{\mu}^{(j)}(u,\alpha,z)
=\left(\frac{\zeta-\alpha^{2}}{f(\alpha,z)}
\right)^{1/4}
\left\{ \mathcal{C}_{\mu}^{(j)}\left(u\zeta^{1/2}\right)A_{\mu}(u,\alpha,z) \right.
\\ \left.
+\frac{\zeta}{u}\frac{\partial
\mathcal{C}_{\mu}^{(j)}\left(u\zeta^{1/2}\right)}{\partial \zeta}
B_{\mu}(u,\alpha,z) \right\},
\end{multline}
where $\gamma_{\mu}^{(0,1)}(u,\alpha)=1$, $\gamma_{\mu}^{(2)}(u,\alpha)=\gamma_{\mu}(u,\alpha)$, $\mathcal{C}_{\mu}^{(0)}=2J_{\mu}$ and $\mathcal{C}_{\mu}^{(1,2)}=H_{\mu}^{(1,2)}$. Here $A_{\mu}(u,\alpha,z)$ and $B_{\mu}(u,\alpha,z)$ are analytic for $z \in Z$, and have the asymptotic expansions \cref{eq84,eq85,eq86,eq87} as $u \to \infty$ uniformly for $z \in Z$ and $\alpha \in [0,\alpha_{0}]$.

\end{theorem}

\section{Associated Legendre functions}
\label{sec:Legendre}
We illustrate the theory by finding asymptotic expansions for solutions of the associated Legendre equation
\begin{equation}
\label{eq97}
\left(t^2-1\right)\frac{d^2 y}{d t^2} +2t\frac{dy}{dt} 
-\left\{\nu(\nu + 1)+\frac{\mu^2}{t^2-1}\right\}y=0.
\end{equation}
Standard solutions are given by $y(u,\alpha,t)=L_{\nu}^{\pm\mu}(t)$ where $L=\mathrm{P},\mathrm{Q},P,Q$ (see \cite[\S14.3]{NIST:DLMF}). For simplicity, we restrict our attention to the half-plane $|\arg(t)|\leq \pi/2$, with extensions to other parts of the complex plane achievable through appropriate connection formulas; see, for example, \cite[\S14.24]{NIST:DLMF}.

To put \cref{eq97} into a form to which the theory applies, let
\begin{equation}
\label{eq98}
z=1-t,
\quad
u = \nu + \tfrac{1}{2}, 
\quad 
\alpha = \frac{\mu}{\nu + \tfrac{1}{2}},
\end{equation}
and
\begin{equation}
\label{eq99}
w(u,\alpha,z)=\left(t^2-1\right)^{1/2}y(u,\alpha,t).
\end{equation}
Then \cref{eq97} takes the desired form \cref{eq01} with
\begin{equation}
\label{eq100}
f(\alpha,z)=\frac{z^{2}-2z+\alpha^{2}}{z^{2}(z-2)^{2}},
\quad
g(z)=-\frac{z^{2}-2z+4}{4z^{2}(z-2)^{2}},
\end{equation}
which agree with \cref{eq02,eq03}.

The zeros of $f$ are the turning points, given by $z=1\pm \sigma$ ($t=\pm \sigma$), where
\begin{equation}
\label{eq101}
\sigma=(1-\alpha^{2})^{1/2},
\end{equation}
and we focus on $z_{t}=1-\sigma$ which lies in the half-plane $\Re(z) \leq 1$ under consideration, and avoid the other one, denoted by $\tilde{z}_{t}=1+\sigma$. Thus
\begin{equation}
\label{eq102}
f(\alpha,z)=\frac{\left(z-z_{t}\right)
\left(z-\tilde{z}_{t}\right)}{z^{2}(z-2)^{2}},
\end{equation}
and in terms of the original variable $t$
\begin{equation}
\label{eq103}
f=\frac{t^{2}-\sigma^{2}}
{\left(t^2-1\right)^{2}}.
\end{equation}
Since $z_t \to 0$ as $\alpha \to 0$ ($\sigma \to 1$), the theory of \cref{sec:main} is directly applicable. Here we assume $0 \leq \alpha \leq 1-\delta$ for arbitrary small positive $\delta$, and hence $z_{t}$ is bounded away from $\tilde{z}_{t}$ (since they coalesce at $z=1$ when $\alpha=1$). In terms of our earlier notation $\alpha_{0}=1-\delta$.

From \cref{eq07,eq100} the Schwarzian in the present case is
\begin{equation} 
\label{eq104}
\psi(\alpha, z) =
-\frac{z(z-2)\left(z^{2}-2z
-4\alpha^{2}(z-1)^{2}+\alpha^{4}\right)}
{4\left(z^{2}-2z+\alpha^{2}\right)^{3}}.
\end{equation}

For the $z-\xi$ mapping we suppose that $z$ lies in the half-plane $\Re(z) \leq 1$ having a cut along $(-\infty,z_{t}]$, equivalently $t$ lying in the half-plane $\Re(t) \geq 0$ having a cut along $[\sigma, \infty)$. With these choices, \cref{eq05} gives
\begin{equation}
\label{eq105}
\xi=\int_{z_{t}}^{z}f^{1/2}(\alpha,v)dv
=\int_{t}^{\sigma}f^{1/2}(\alpha,1-v)dv,
\end{equation}
with the roots chosen so that $i\xi$ is positive for $t \in [0,\sigma]$ (i.e. $z \in [z_{t},1]$), and by continuity elsewhere in the cut half-planes. Hence, from \cref{eq100,eq101}
\begin{equation}
\label{eq106}
\xi= \alpha\,\mathrm{arctanh}
\left\{\frac{\left(t^{2}-\sigma^{2}\right)^{1/2}}
{\alpha t}\right\}
- \mathrm{arccosh}\left(\frac{t}{\sigma}\right).
\end{equation}
In the real interval $0 \leq t \leq \sigma$ it is seen that $\xi$ is purely imaginary, with explicit representation
\begin{equation}
\label{eq107}
\xi=i\alpha\arccos\left\{\frac{\alpha t}
{\sigma\left(1-t^{2}\right)^{1/2}}\right\}
-i\arccos\left(\frac{t}{\sigma}\right),
\end{equation}
and with our choice of branches we obtain from \cref{eq106} the limiting behavior
\begin{equation} 
\label{eq108}
\xi=\mp\ln\left(\frac{2t}{\sigma}\right)
\mp\alpha\ln\left(\frac{\sigma}{1+\alpha}\right)
+\frac{1}{2}\alpha\pi i
\quad (t \to \infty \pm i0).
\end{equation}

Similarly to \cref{eq55}, we introduce a variable that aids in the integrals used to evaluate the LG coefficients, and specifically we define
\begin{equation}
\label{eq109}
\beta=\left(\frac{t-\sigma}{t+\sigma}\right)^{1/2}
=\left(\frac{z-z_{t}}{z-\tilde{z}_{t}}\right)^{1/2},
\end{equation}
with the roots chosen so that $i\beta$ is positive in $z_{t}<z \leq 1$ and by continuity elsewhere in the $z$ half-plane $\Re(z) \leq 1$ having a cut along $(-\infty,z_{t}]$. Now, from \cref{sec:Introduction}, in general $\xi_{1,2}=\pm \infty+\alpha\pi i/2$, therefore, from \cref{eq98,eq108}, in the present case we have $z_{1,2}=-\infty \pm i0$. Thus, from \cref{eq109}
\begin{equation}
\label{eq110}
\lim_{z \to z_{1,2}}\beta(\alpha,z)=\pm 1.
\end{equation}
Note that from \cref{eq98,eq109}
\begin{equation}
\label{eq111}
1-z=t=\frac{\sigma\left(\beta^{2}+1\right)}{1-\beta^{2}}.
\end{equation}

Now from \cref{eq05,eq100,eq101,eq106,eq109,eq111}
\begin{equation}
\label{eq112}
G(\alpha,\beta):=\frac{d\beta}{d\xi}
=\frac{1}{f^{1/2}(\alpha,z)}\frac{d\beta}{dz}
=\frac{\left(\beta^{2}-1\right)\left\{
\alpha^{2}\left(\beta^{2}+1\right)^{2}-4\beta^{2}\right\}}
{8\,\sigma^{2}\beta^{2}}.
\end{equation}
Then, from \cref{eq23,eq25,eq100,eq101,eq105,eq109,eq110,eq111,eq112}, the first LG coefficient, regarded as a function of $\beta$, is given by
\begin{multline}
\label{eq113}
E_{1}(\alpha,\beta)
=\int_{1}^{\beta}
\frac{\left(p^{2}-1\right)
\left\{4p^{2}-\alpha^{2}
\left(5p^{4}+6p^{2}+5\right)\right\}}
{64 \sigma^{2} p^{4}}dp
+\lambda_{1}(\alpha)
\\
=-\frac{\left(\beta-1\right)^{2}
\left\{5\alpha^{2}\left(\beta^{2}+\beta+1\right)^{2}
-3\left(4-\alpha^{2}\right)\beta^{2}\right\}}
{192 \sigma^{2}\beta^{3}}
+\lambda_{1}(\alpha)
\\
=-\frac{1}{192\,\sigma^{2}}
\left\{\frac{5\alpha^{2}}{\beta^{3}}
-\frac{3\left(4-\alpha^{2}\right)}{\beta}
-3(4-\alpha^{2})\beta
+5\alpha^{2}\beta^{3}
\right\}
-\frac{3-2\alpha^{2}}{24\sigma^{2}}
+\lambda_{1}(\alpha).
\end{multline}

To determine $\lambda_{1}(\alpha)$, and $\lambda_{2s+1}(\alpha)$ ($s=0,1,2,\ldots$) in general, from \cref{eq29,eq110} the following identities must hold
\begin{multline}
\label{eq114}
\lambda_{2s+1}(\alpha)
=\lim_{z \to z_{1}}E_{2s+1}(\alpha,\beta)
=E_{2s+1}(\alpha,1)
\\
=-\lim_{z \to z_{2}}E_{2s+1}(\alpha,\beta)
=-E_{2s+1}(\alpha,-1)
\quad (s=0,1,2,\ldots),
\end{multline}
and specifically $E_{2s+1}(\alpha,1)=-E_{2s+1}(\alpha,-1)$. Thus the choice of integration constants given generally by \cref{eq26} for \cref{eq25} are equivalent to choosing $\lambda_{2s+1}(\alpha)$ so that the coefficients $E_{2s+1}(\alpha,\beta)$ are odd Laurent polynomials, i.e. having no additive constant terms. Thus, from the third equality in \cref{eq113} we see that 
\begin{equation}
\label{eq115}
\lambda_{1}(\alpha)=\frac{3-2\alpha^{2}}{24\sigma^{2}},
\end{equation}
and consequently from \cref{eq113,eq115}
\begin{equation}
\label{eq116}
E_{1}(\alpha,\beta)
=\frac{\left(\beta^{2}+1\right)
\left\{2 \left(\alpha^{2}+6\right)\beta^{2}
-5\alpha^{2}\left(\beta^{4}+1\right)\right\}}
{192 \sigma^{2} \beta^{3}}.
\end{equation}

The even coefficients are more straightforward, since from \cref{eq28} they are simply chosen to vanish at $z=z_{1,2}$, or equivalently for $\beta=\pm 1$ (see \cref{eq110}). From \cref{eq23,eq25,eq100,eq104,eq111,eq112} we then find for the first one that
\begin{equation}
\label{eq117}
E_{2}(\alpha,\beta)
=\frac{\left(\beta^{2}-1\right)^{2}}
{1024 \, \sigma^{4} \beta^{6}}
\left\{4\beta^{2}-\alpha^{2}
\left(\beta^{2}+1\right)^{2}\right\}
\left\{4\beta^{2}-\alpha^{2}
\left(5\beta^{4}+6\beta^{2}+5\right)\right\}.
\end{equation}
From \cref{eq24,eq25,eq100,eq111,eq112} subsequent terms (odd and even) are given by
\begin{multline} 
\label{eq118}
E_{s+1}(\alpha,\beta)=
-\frac{1}{2} G(\alpha,\beta)
\pdv{E_{s}(\alpha,\beta)}{\beta}
\\
-\frac{1}{2}  \int_{(s)}
G(\alpha,\beta)
\sum\limits_{j=1}^{s-1}
\pdv{E_{j}(\alpha,\beta)}{\beta}
\pdv{E_{s-j}(\alpha,\beta)}{\beta} d\beta
\quad (s=2,3,4,\ldots),
\end{multline}
where
\begin{equation}
\label{eq119}
\int_{(s)}H(\beta)d\beta
=
    \begin{cases}
        \int_{1}^{\beta}H(p)dp & \text{if } s \,\, \text{is odd}\\
        \frac{1}{2}\left[ \int_{\beta_{0}}^{\beta}H(p)dp
        +\int_{-\beta_{0}}^{\beta}H(p)dp \right]
        & \text{if } s \,\, \text{is even}
    \end{cases}.
\end{equation}
Here $\beta_{0}$ is an arbitrary nonzero constant, and the  integration paths avoid the pole at $p=0$.

From \cref{eq112,eq116,eq117}, and induction, the first term on the RHS of \cref{eq118} is, respectively, an even/odd Laurent polynomial when $s$ is odd / even, and vanishes at $\beta=\pm 1$. Likewise, the integrand of the second term on the RHS is, respectively, an even/odd Laurent polynomial when $s$ is odd / even. Thus, from \cref{eq119} we have that $E_{2s+1}(\alpha,\beta)$ ($s=0,1,2,\ldots$) are Laurent polynomials in $\beta$ which are odd, which, as noted earlier, is required. Moreover, $E_{2s}(\alpha,\beta)$...
\begin{equation}
\label{eq120}
\lim_{z \to z_{1,2}}E_{2s}(\alpha,\beta)
=E_{2s}(\alpha,\pm 1)=0,
\end{equation}
in accord with \cref{eq28}. 

Note that if $H(p)$ is an odd Laurent polynomial having no $p^{-1}$ term, then $\int_{-1}^{1}H(p)dp \allowbreak =0$ along any path avoiding $p=0$. That this is true here, and hence, in particular, no $\ln(\beta)$ terms can appear upon integration for the even coefficients in \cref{eq118} ($s$ odd), can be inferred from \cref{eq27} that they satisfy the asymptotic expansion
\begin{equation}
\label{eq121}
\sum\limits_{s=1}^{\infty }
\frac{E_{2s}(\alpha,\beta) }{u^{2s}}
\sim -\frac{1}{2}\ln \left\{ 1+G(\alpha,\beta)
\sum\limits_{s=1}^{\infty}
\frac{1}{u^{2s}}
\pdv{E_{2s-1}(\alpha,\beta)}{\beta}\right\}
\quad (u \to \infty).
\end{equation}

As an aside, we observe that all the coefficients, odd and even, are palindromic, i.e., they are invariant under the map $\beta \to 1/\beta$, which can readily be verified by induction. This is due to the symmetry of the differential equation \cref{eq97} about the origin and the definition \cref{eq109} of $\beta$.

Next, following \cite[Eq. (5.11)]{Dunster:2025:LCT}, define
\begin{equation}
\label{eq122}
\mathsf{Q}^{\mu}_{\nu,1}(t)
=
    \begin{cases}
        \boldsymbol{Q}^{\mu}_{\nu}(t) & \Im(t) \geq 0\\
        \boldsymbol{Q}^{\mu}_{\nu,1}(t) & \Im(t) < 0
    \end{cases}\, , \quad
    \mathsf{Q}^{\mu}_{\nu,-1}(t)
=\overline{\mathsf{Q}^{\mu}_{\nu,1}(\bar{t})}.
\end{equation}
From this we note that \cite[Eq.~14.9.14]{NIST:DLMF} $\mathsf{Q}^{-\mu}_{\nu,\pm 1}(t)=\mathsf{Q}^{\mu}_{\nu,\pm 1}(t)$. The significance of these solutions is that they are analytic in the $z$ plane having cuts along $(-\infty,-1]$ and $[1,\infty)$, with $\mathsf{Q}^{\mu}_{\nu,\pm 1}(t)$ recessive at $t = \infty \pm i0$ ($z = -\infty \mp i0$).

From \cite[Eqs.~14.3.1 and 14.8.1]{NIST:DLMF} the solution of \cref{eq97} that is recessive at $t=1$ is the Ferrers function given by
\begin{equation} 
\label{eq123}
\mathsf{P}_{\nu}^{-\mu}(t)
=\left(\frac{1-t}{1+t}\right)^{\mu/2}
\mathbf{F}\left(\nu+1,-\nu;1+\mu;\tfrac{1}{2}
-\tfrac{1}{2}t\right),
\end{equation}
where $\mathbf{F}$ is the scaled hypergeometric function defined by \cite[Eq. 15.2.2]{NIST:DLMF}. In this, we can replace $\mu \neq 1,2,3,\ldots$ by $-\mu$ to obtain the solution of \cref{eq97} that has the exponent $-\mu/2$ at the singularity $t=1$. From \cref{eq122} \cite[Eq. 14.23.2]{NIST:DLMF} the following connection formula will be utilized
\begin{equation}
\label{eq124}
\mathsf{P}^{\pm \mu}_{\nu}(t)=\frac{i\Gamma(\nu \pm \mu+1)}{\pi}
\left\{e^{ \mp \mu\pi i/2}\mathsf{Q}^{\mu}_{\nu,1}(t)
-e^{\pm \mu\pi i/2}\mathsf{Q}^{\mu}_{\nu,-1}(t)\right\}.
\end{equation}

We now match solutions of \cref{eq97} that are recessive at $t=\infty\pm i 0$ with the LG expansions given by \cref{eq30,eq31}. To this end, from \crefcomma{eq108,eq122} and \cite[Eq.~14.8.15]{NIST:DLMF} for nonnegative $\nu$
\begin{equation}
\label{eq125}
\mathsf{Q}^{\mu}_{\nu,\pm 1}(t)
=\frac{\sqrt{\pi}}
{\Gamma\left(\nu+\frac{3}{2}\right)
(2t)^{\nu+1}}
\left\{1+\mathcal{O}
\left(\frac{1}{t}\right)\right\}
\quad \left(t \rightarrow \infty \pm i0 \right).
\end{equation}
Thus, on referring to \cref{eq28,eq29,eq30,eq31,eq98,eq101,eq108,eq125}, we arrive at the desired representations
\begin{multline} 
\label{eq126}
\mathsf{Q}^{\mu}_{\nu,\mp 1}(t)
=\sqrt{\frac{\pi}{2}}
\left(\frac{\sigma}{1+\alpha}\right)^{\mu}
\frac{e^{\pm(\mu+1)\pi i/2}}
{\sigma^{u}\Gamma(u+1)}
\\ \times
\left\{\frac{\zeta-\alpha^{2}}
{\left(z-z_{t}\right)
\left(\tilde{z}_{t}-z\right)\zeta^{2}}\right\}^{1/4}
\tilde{W}_{\mu}^{(1,2)}(u,\alpha,z).
\end{multline}
The quarter power factors in front of $\tilde{W}_{\mu}^{(2,1)}(u,\alpha,z)$ come from \cref{eq11,eq98,eq99,eq103}.

In these expansions, the relative error terms $\eta_{n}^{(1,2)}(u,\alpha,z)$ given in \cref{eq30,eq31} are $\mathcal{O}(u^{-n})$ as $u \to \infty$ for $z \in Z_{\infty}$, which in our case is an unbounded subset of the half-plane $|\arg(1-z)|\leq \pi/2$ that does not contain $z=z_{t}$ or $z=0$. We are not concerned with specifying the region precisely, since our main expansions below involving Bessel functions will follow directly from \cref{thm:main} and be valid in the entire half-plane.

Next, we have the identification
\begin{multline} 
\label{eq127}
\mathsf{P}^{-\mu}_{\nu}(t)
=\sqrt{\frac{1}{2\pi}}
\left(\frac{\sigma}{1+\alpha}\right)^{\mu}
\frac{\Gamma\left(u - \mu+\tfrac12\right)}
{\sigma^{u}\Gamma(u+1)}
\\ \times
\left\{\frac{\zeta-\alpha^{2}}
{\left(z-z_{t}\right)
\left(\tilde{z}_{t}-z\right)\zeta^{2}}\right\}^{1/4}
\tilde{W}_{\mu}^{(0)}(u,\alpha,z),
\end{multline}
since both sides are solutions of \cref{eq97} that are recessive at $t=1$ ($z=0$). The constant of proportionality, along with identifying that $\gamma_{\mu}(u,\alpha)=1$, follow from \cref{eq39,eq124,eq126}. Note that this value of $\gamma_{\mu}(u,\alpha)$ is in agreement with \cref{eq42}. Also observe that \cref{eq124}, with upper signs taken, together with \cref{eq126}, agrees with \cref{eq41} with $\gamma_{\mu}(u,\alpha)=1$.

Now, using \cref{thm:main} and \cref{eq76,eq77,eq126}, we obtain the desired expansion
\begin{multline} 
\label{eq128}
\mathsf{Q}^{\mu}_{\nu,\mp 1}(t)
=\sqrt{\frac{\pi}{2}}
\left(\frac{\sigma}{1+\alpha}\right)^{\mu}
\frac{e^{\pm(\mu+1)\pi i/2}}
{\sigma^{u}\Gamma(u+1)}
\left(\frac{\zeta-\alpha^{2}}
{\sigma^{2}-t^{2}}\right)^{1/4}
\\ \times
\left\{H_{\mu}^{(1,2)}\left(u\zeta^{1/2}\right)A_{\mu}(u,\alpha,z)
+\frac{\zeta}{u}\frac{\partial
H_{\mu}^{(1,2)}\left(u\zeta^{1/2}\right)}{\partial \zeta}
B_{\mu}(u,\alpha,z)\right\},
\end{multline}
and similarly, using \cref{eq88,eq127}, we obtain
\begin{multline} 
\label{eq129}
\mathsf{P}^{-\mu}_{\nu}(t)
=\sqrt{\frac{2}{\pi}}
\left(\frac{\sigma}{1+\alpha}\right)^{\mu}
\frac{\Gamma\left(u - \mu+\tfrac12\right)}
{\sigma^{u}\Gamma(u+1)}
\left(\frac{\zeta-\alpha^{2}}
{\sigma^{2}-t^{2}}\right)^{1/4}
\\ \times
\left\{J_{\mu}\left(u\zeta^{1/2}\right)A_{\mu}(u,\alpha,z)
+\frac{\zeta}{u}\frac{\partial
J_{\mu}\left(u\zeta^{1/2}\right)}{\partial \zeta}
B_{\mu}(u,\alpha,z)\right\}.
\end{multline}
The expansions \cref{eq84,eq85,eq86,eq87} for $A_{\mu}(u,\alpha,z)$ and $B_{\mu}(u,\alpha,z)$ are uniformly valid for $|\arg(1-z)|\leq \pi/2$, i.e. $|\arg(t)|\leq \pi/2$.

From \cite[Eq. (5.14)]{Dunster:2025:LCT} the numerically satisfactory companion Ferrers function to $\mathsf{P}_{\nu}^{-\mu}(t)$ for $0 \leq t < 1$ is the real-valued function given by
\begin{equation} 
\label{eq130}
\mathsf{Q}_{\nu}^{-\mu}(t)
=\tfrac12 \Gamma(\nu-\mu+1)\left\{e^{\mu \pi i/2}
\mathsf{Q}^{\mu}_{\nu,1}(t)
+e^{-\mu \pi i/2}\mathsf{Q}^{\mu}_{\nu,-1}(t)\right\}.
\end{equation}
Consequently, on referring to \cite[Eq.~10.4.4]{NIST:DLMF}, we find the following representation for $0 \leq t <1$
\begin{multline} 
\label{eq131}
\mathsf{Q}^{-\mu}_{\nu}(t)
=-\sqrt{\frac{\pi}{2}}
\left(\frac{\sigma}{1+\alpha}\right)^{\mu}
\frac{\Gamma\left(u - \mu+\tfrac12\right)}
{\sigma^{u}\Gamma(u+1)}
\left(\frac{\zeta-\alpha^{2}}
{\sigma^{2}-t^{2}}\right)^{1/4}
\\ \times
\left\{Y_{\mu}\left(u\zeta^{1/2}\right)A_{\mu}(u,\alpha,z)
+\frac{\zeta}{u}\frac{\partial
Y_{\mu}\left(u\zeta^{1/2}\right)}{\partial \zeta}
B_{\mu}(u,\alpha,z)\right\}.
\end{multline}

Finally, consider solutions to \cref{eq97} that are real for $t=x \in (1,\infty)$ (thus, from \cref{eq12,eq111} $z=1-x$ and $\zeta$, both lie in $(-\infty,0)$). For the one that is recessive at $x=1$ we have from \cite[Eq. 14.23.4]{NIST:DLMF}
\begin{equation} 
\label{eq132}
P^{-\mu}_{\nu}(x)
=e^{\mu\pi i/2}\mathsf{P}^{-\mu}_{\nu}(x+i0).
\end{equation}
Now $\zeta^{1/2}=-i|\zeta|^{1/2}$ for $t=x+i0$ and $1<x<\infty$. Thus, from \cref{eq129} and \cite[Eq.~10.27.6]{NIST:DLMF}
\begin{multline} 
\label{eq133}
P^{-\mu}_{\nu}(x)
=\sqrt{\frac{2}{\pi}}
\left(\frac{\sigma}{1+\alpha}\right)^{\mu}
\frac{\Gamma\left(u - \mu+\tfrac12\right)}
{\sigma^{u}\Gamma(u+1)}
\left(\frac{|\zeta|+\alpha^{2}}
{x^{2}-\sigma^{2}}\right)^{1/4}
\\ \times
\left\{I_{\mu}\left(u|\zeta|^{1/2}\right)A_{\mu}(u,\alpha,1-x)
+\frac{|\zeta|}{u}\frac{\partial
I_{\mu}\left(u|\zeta|^{1/2}\right)}{\partial |\zeta|}
B_{\mu}(u,\alpha,1-x)\right\},
\end{multline}
with the expansions being uniformly valid for $1<x<\infty$. 

Using \crefcomma{eq122,eq128} and \cite[Eq.~10.27.8]{NIST:DLMF}, for the solution that is real-valued for $x \in (1,\infty)$ and recessive at $x=\infty$, we similarly obtain
\begin{multline} 
\label{eq134}
\mathbf{Q}^{\mu}_{\nu}(x)
=\sqrt{\frac{2}{\pi}}
\left(\frac{\sigma}{1+\alpha}\right)^{\mu}
\frac{1}{\sigma^{u}\Gamma(u+1)}
\left(\frac{|\zeta|+\alpha^{2}}
{x^{2}-\sigma^{2}}\right)^{1/4}
\\ \times
\left\{K_{\mu}\left(u|\zeta|^{1/2}\right)A_{\mu}(u,\alpha,1-x)
+\frac{|\zeta|}{u}\frac{\partial
K_{\mu}\left(u|\zeta|^{1/2}\right)}{\partial |\zeta|}
B_{\mu}(u,\alpha,1-x)\right\}.
\end{multline}

\subsection{Numerical Example}
\label{numerics}

We now illustrate the accuracy of our asymptotic expansion for $\mathsf{P}^{-\mu}_{\nu}(t)$. For $t\in[0,1)$, define the approximation $\mathsf{P}_{4}(\nu,\mu,t)$ by taking the first four terms of the series in \cref{eq84,eq85,eq129}, that is.
\begin{multline} 
\label{eq135}
\mathsf{P}_{4}(\nu,\mu,t)
=\left(\frac{\sigma}{1+\alpha}\right)^{\mu}
\frac{\sqrt{u}\,\Gamma\left(u - \mu+\tfrac12\right)}
{\sigma^{u}\Gamma(u+1)}
\left(\frac{\zeta-\alpha^{2}}
{\sigma^{2}-t^{2}}\right)^{1/4}
\exp \left\{\sum\limits_{s=0}^{3}
\frac{\lambda_{2s+1}(\alpha)}{u^{2s+1}}\right\}
\\ \times
\left[J_{\mu}\left(u\zeta^{1/2}\right)
\left\{1 + \sum_{s=1}^{3}
\frac{\mathrm{A}_{s}(\alpha,\beta)}{u^{2s}}\right\}
+\frac{2\zeta}{u^{2}}\frac{\partial
J_{\mu}\left(u\zeta^{1/2}\right)}{\partial \zeta}
\sum_{s=0}^{3}
\frac{\mathrm{B}_{s}(\alpha,\beta)}{u^{2s}}\right].
\end{multline}
It is convenient to express the coefficients in terms of $\beta$ (defined in \cref{eq109}) rather than $z=1-t$, because the LG coefficients in \cref{eq116,eq117,eq118,eq119} are Laurent polynomials in $\beta$.

To compute $\zeta$ ($\ge \alpha^{2}$) for $0 \leq t\le \sigma$, we numerically solve the following equation obtained from \cref{eq51,eq107}:
\begin{equation} 
\label{eq136}
\left(\zeta-\alpha^{2}\right)^{1/2}
-\alpha \arccos\left(\frac{\alpha}{\zeta^{1/2}}\right)
=\arccos\left(\frac{t}{\sigma}\right)
-\alpha\arccos\left\{\frac{\alpha t}
{\sigma\left(1-t^{2}\right)^{1/2}}\right\}.
\end{equation}
Likewise, for $\sigma \le t < 1$, we use \cref{eq19,eq106} to solve for $\zeta \in (0,\alpha^{2}]$ from the implicit equation
\begin{multline} 
\label{eq137}
\alpha\ln\left\{\left(\alpha^{2}
-\zeta\right)^{1/2}+\alpha\right\}
-\left(\alpha^{2}-\zeta\right)^{1/2}
-\frac{1}{2}\alpha\ln(\zeta)
\\
=\alpha\,\mathrm{arctanh}
\left\{\frac{\left(t^{2}-\sigma^{2}\right)^{1/2}}
{\alpha t}\right\}
- \mathrm{arccosh}\left(\frac{t}{\sigma}\right).
\end{multline}

Because $\mathsf{P}^{-\mu}_{\nu}(t)$ has zeros on $[0,1)$, an envelope function is introduced to assess the relative error in this interval. Following \cite[\S 6]{Dunster:2025:LCT}, therefore define
\begin{equation}
\label{eq138}
M(\nu,\mu,t)=
    \begin{cases}
        \left[\left\{\mathsf{P}^{-\mu}_{\nu}(t)\right\}^{2}
        +\left\{2\mathsf{Q}^{-\mu}_{\nu}(t)/\pi
        \right\}^{2}\right]^{1/2} 
        & (0<t \leq q^{-\mu}_{\nu,1})\\
        \mathsf{P}^{-\mu}_{\nu}(t) & (q^{-\mu}_{\nu,1} < t < 1)
    \end{cases},
\end{equation}
where $q^{-\mu}_{\nu,1}$ is the largest positive zero of $\mathsf{Q}^{-\mu}_{\nu}(t)$. On $(0,q^{-\mu}_{\nu,1})$ this positive function closely tracks the amplitude of $\mathsf{P}^{-\mu}_{\nu}(t)$, while on $(q^{-\mu}_{\nu,1},1)$ it coincides with $\mathsf{P}^{-\mu}_{\nu}(t)$.

\begin{figure}
 \centering
 \includegraphics[
 width=0.9\textwidth,keepaspectratio]{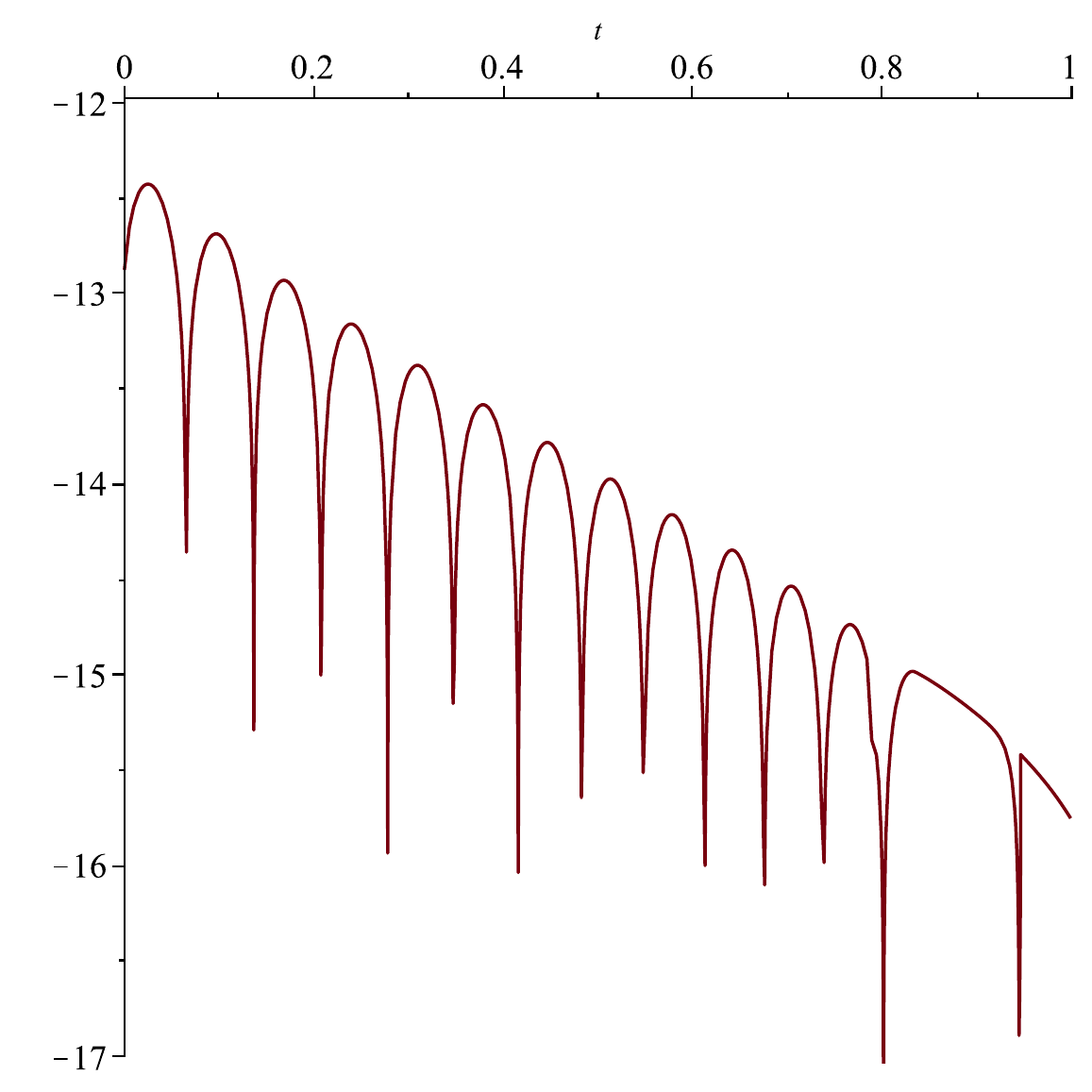}
 \caption{Graph of $\Omega_{4}(\nu,\mu,t)$ for $\nu=50$ and $\mu=u \alpha= (\nu+\frac12) \alpha$, with $\alpha=0.5$.}
 \label{fig:numerics1}
\end{figure}

The following relative error function was then computed
\begin{equation} 
\label{eq139}
\Omega_{4}(\nu,\mu,t)
=\log_{10}\left\{\frac{\left|\mathsf{P}^{-\mu}_{\nu}(t)
-\mathsf{P}_{4}(\nu,\mu,t)\right|}
{M(\nu,\mu,t)}\right\}
\quad (0 \leq t <1).
\end{equation}
A plot of $\Omega_{4}(\nu,\mu,t)$ with $\nu=50$ ($u=50.5$) and $\mu=u\alpha$ is given in \cref{fig:numerics1} for $\alpha=0.5$, noting that the value of the turning point is at $t=\sigma = \sqrt{3}/2=0.86602\cdots$, which is close to the singularity $t=1$. The plot demonstrates high accuracy uniformly throughout the interval $[0,1)$. We remark that the cusps are at points where the relative error is exactly zero. 

Similarly, in \cref{fig:numerics2} we plot $\Omega_{4}(\nu,\mu,t)$ for $\nu=50$ and $\mu=u\alpha$ with $\alpha=0.9$. Here, the turning point occurs at $\sigma=\sqrt{19}/10 = 0.43589\cdots$, which lies farther from the singularity. The accuracy throughout the interval remains good, though it is lower than in the case where the turning point is closer to the singularity. Compared to the previous parameter values, the turning points at $t=\pm\sigma$ are closer to each other. Since the present approximation is not uniform in a neighborhood of $t=-\sigma$, this increased proximity explains the reduction in overall accuracy. In such cases, a uniform expansion based on the new theory of coalescing turning points can yield a sharper approximation throughout the interval without increasing the number of terms; see \cite[\S 5]{Dunster:2025:LCT}.

\begin{figure}
 \centering
 \includegraphics[
 width=0.9\textwidth,keepaspectratio]{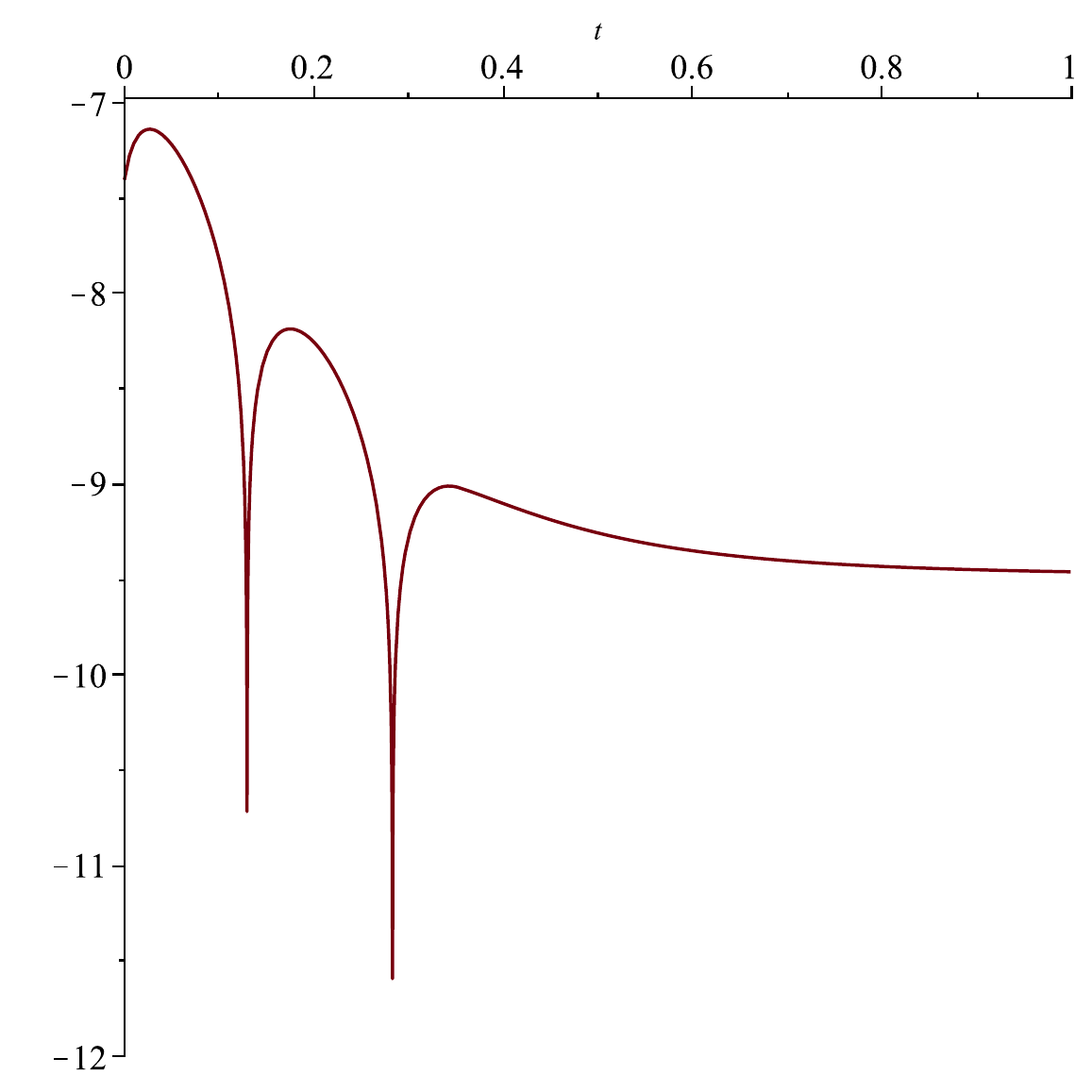}
 \caption{Graph of $\Omega_{4}(\nu,\mu,t)$ for $\nu=50$ and $\mu=u \alpha= (\nu+\frac12) \alpha$, with $\alpha=0.9$.}
 \label{fig:numerics2}
\end{figure}

Finally, we mention that the coefficients in \cref{eq135}, while bounded at $t=\sigma$, have a removable singularity there and, therefore, are not numerically stable to evaluate when $t$ is close to this value. Accordingly, in our computations we expanded the coefficients as a Maclaurin series in $\beta$ for $|t-\sigma|<0.08$, and computed them directly for all other values of $t$. In doing so, we used from \cref{eq55,eq109,eq137},
\begin{multline} 
\label{eq140}
\hat{\beta}=\kappa\beta+\frac{\left(3+\sigma^{2}\right)\kappa
-8\sigma^{2}}{5\alpha^{2}}\beta^{3}
\\
+\frac{3\kappa^{5}+(\sigma^{2}+3)
\left\{(43\sigma^{2}+29)\kappa-168\sigma^{2}\right\}}
{175\alpha^{4}}\beta^{5}
+\ldots
\quad (\beta \to 0, \,\alpha>0),
\end{multline}
where $\kappa=2\sigma^{2/3}$; the limiting value as $\alpha \to 0$ also applies, namely
\begin{equation} 
\label{eq141}
\hat{\beta}=2\beta+\tfrac{2}{3}\beta^{3}
+\tfrac{2}{3}\beta^{5}+\ldots
\quad (\beta \to 0, \,\alpha=0).
\end{equation}

From this, corresponding series for the coefficients can be constructed. For example, as $\beta \to 0$ ($\alpha>0$),
\begin{equation} 
\label{eq142}
\mathrm{B}_{0}(\alpha,\beta)
=\frac{2\sigma^{2}\kappa+(\sigma^{2}+3)
(3-4\sigma^{2})}{140\alpha^{2}\sigma^{2}\kappa}
+\mathcal{O}(\beta),
\end{equation}
and
\begin{equation} 
\label{eq143}
\mathrm{A}_{1}(\alpha,\beta)
=\frac{9(\sigma^{2}+3)(3-4\sigma^{2})
\kappa^{2}+224\sigma^{6}-276\sigma^{4}
+588\sigma^{2}-392}{50400\alpha^{2}\sigma^{4}}
+\mathcal{O}(\beta),
\end{equation}
with again the limiting values being applicable when $\alpha = 0$. Because the higher-order terms are unwieldy in exact form, it is more computationally efficient to evaluate them numerically once (for the parameter values used in the plot) and to store these values for subsequent use.

\section*{Acknowledgment}
Financial support from Ministerio de Ciencia e Innovación pro\-ject PID2024-159583NB-I00 (MICIU/ AEI / 10.13039/501100011033 / FEDER, UE) is acknowledged.

\section*{Conflict of interest}
The author declares no conflicts of interest.

\makeatletter
\interlinepenalty=10000

\bibliographystyle{siamplain}
\bibliography{biblio}
\end{document}